\documentclass[a4paper,10pt,twoside]{amsart}
\usepackage{amsfonts, amsmath, amssymb, appendix, mathrsfs, esint}%, mathabx}
\usepackage{enumerate}
\usepackage{enumitem}
\usepackage{graphicx}
\numberwithin{equation}{section}
\usepackage{hyperref}

\makeatletter
\DeclareRobustCommand*{\bfseries}{%
  \not@math@alphabet\bfseries\mathbf
  \fontseries\bfdefault\selectfont
  \boldmath
}
\makeatother

\def\dto{{\scriptstyle\buildrel{\scriptstyle\longrightarrow}
\over{{}_{\scriptstyle\longrightarrow}}}}
\def\wto{\rightharpoonup}
\def\swto{\stackrel{*}{\wto}}
\def\RR{\mathbb{R}}

\def\a{\W}

\def\MM{\mathbb{M}}
\def\NN{\mathbb{N}}
\def\Z{{\mathcal Z}}

\def\Q{\mathrm Q}

\def\Cub{\mathrm{Cub}}

\DeclareMathOperator{\maxw}{\mathcal{M}_{0}}

\DeclareMathOperator*{\esssup}{ess\,sup}

\DeclareMathOperator*{\limin}{\underline{\lim}}
\DeclareMathOperator*{\limsu}{\overline{\lim}}

\DeclareMathOperator{\subsubset}{\subset\!\subset}
\DeclareMathOperator{\ssetminus}{\!\setminus\!}

\DeclareMathOperator{\nspace}{\!\!\!\!\!\!\!\!\!}

\def\ssup{\!\!>\!\!}
\def\sinf{\!\!<\!\!}

\def\p{{\mathrm p}}
\def\M{\mathrm{m}}

\def\L{\mathcal{L}}

\def\diam{{\rm diam}}
\def\dom{{\rm dom}}
\def\dist{{\rm dist}}
\def\F{\overline{F}}
\def\FF{\mathcal{F}}

\def\inte{{\rm int}}

\def\eps{\varepsilon}

\def\O{\mathcal O}

\def\G{\mathcal G}
\def\S{\mathcal S}

\def\W{{f}}

\def\sgn{\textrm{sgn}}

\DeclareMathAlphabet{\mathpzc}{OT1}{pzc}{m}{it}

\newtheorem{theorem}{Theorem}[section]
\newtheorem{lemma}{Lemma}[section]

\newtheorem{proposition}{Proposition}[section]
\newtheorem{corollary}{Corollary}[section]
\newtheorem{definition}{Definition}[section]

\renewcommand{\liminf}{\limin}
\renewcommand{\limsup}{\limsu}
\theoremstyle{remark}
\newtheorem{remark}{Remark}[section]

\newlist{hyp}{enumerate}{1}
\setlist[hyp,1]{label={\rm (${\rm A}_{\arabic*}$)}}

\newlist{hypg}{enumerate}{1}
\setlist[hypg,1]{label={\rm (${\rm B}_{\arabic*}$)}}

\newlist{hypc}{enumerate}{1}
\setlist[hypc,1]{label={\rm (${\rm C}_{\arabic*}$)}}

\newlist{hypgr}{enumerate}{1}
\setlist[hypgr,1]{label={\rm (${\rm D}_{\arabic*}$)}}

\numberwithin{equation}{section}

\title[]{On the relaxation of unbounded multiple integrals}

\author[Omar Anza Hafsa]{Omar Anza Hafsa}
\address{UNIVERSITE de NIMES, Laboratoire MIPA, Site des Carmes, Place Gabriel P\'eri, 30021 N\^\i mes, France}
\address{Laboratoire LMGC, UMR-CNRS 5508, Place Eug\`ene Bataillon, 34095 Montpellier, France.}
\email{Omar Anza Hafsa <omar.anza-hafsa@unimes.fr>}

\author[Jean Philippe Mandallena]{Jean Philippe Mandallena}
\address{UNIVERSITE de NIMES, Laboratoire MIPA, Site des Carmes, Place Gabriel P\'eri, 30021 N\^\i mes, France}
\email{Jean-Philippe Mandallena <jean-philippe.mandallena@unimes.fr>}

\begin{document}
\begin{abstract} We study the relaxation of multiple integrals of the calculus of variations where the integrands are nonconvex with convex effective domain and can take the value $\infty$. We use local techniques based on measure arguments to prove integral representation in Sobolev spaces of functions which are almost everywhere differentiable. Applications are given in the scalar case and in the case of integrands with quasiconvex growth and $\p(x)$-growth.
\end{abstract}
\maketitle
%\setcounter{tocdepth}{1}
%\tableofcontents
\section{Introduction} 
Let $m,d\ge 1$ be two integers. Let $\Omega\subset\RR^d$ be a nonempty bounded open set with Lipschitz boundary. Define $F:W^{1,p}(\Omega;\RR^m)\to [0,\infty]$ by
\[
F(u):=\int_\Omega \W(x,\nabla u(x))dx
\]
where the integrand $\W:\Omega\times \MM^{m\times d}\to [0,\infty]$ is Borel measurable, and $\MM^{m\times d}$ stands for the set of $m$-rows and $d$-columns matrices. The ``relaxed" functional $\F$ is given by
\begin{align*}
\F(u):=\inf\left\{\liminf_{n\to \infty} F(u_n):W^{1,p}(\Omega;\RR^m)\ni u_n\wto u\right\}
\end{align*}
(if $p=\infty$ then replace $\wto$ by $\stackrel{\ast}{\wto}$).
The goal of the paper is to study the integral representation of $\F$ for nonconvex integrands $\W$ which can take the value $\infty$. In this case, the effective domain $\dom\W(x,\cdot):=\{\xi\in\MM^{m\times d}:\W(x,\xi)\sinf\infty\}$ of $f(x,\cdot)$ is the natural set of constraints for the gradients, the interest of such constrained relaxation problems is well described in the book \cite{carbone-dearcangelis02}. 

In the scalar case, i.e., when $\min\{d,m\}=1$, the integral representation of $\F$ is studied in \cite{dearcangelis-monsurro-zappale04,dearcangelis-zappale05, zappale05}. Under convexity of $\dom\W(x,\cdot)$ and some regularity properties of the multifunction $x\mapsto \dom\W(x,\cdot)$, integral representations with convexification of $\W(x,\cdot)$ are obtained. The present work focus on the vectorial case, i.e., when $\min\{d,m\}\ssup 1$, in this context few is known, particularly the quasiconvexification process when the integrand $\W$ is not finite is not yet understood (for works in this direction see for instance \cite{benbelgacem00,oah-jpm07,oah-jpm08a,oah10}). The main difficulty in the integral representation of $\F$ is that usually we use an approximation result of functions of $W^{1,p}(\Omega;\RR^m)$ by more regular ones, usually continuous piecewise affine or continuously differentiable functions, and this choice implies different relaxed functionals. This situation is known as Lavrentiev phenomenon (or gap) (see for instance \cite{buttazzo-belloni95}). But it is not known whether such approximation results exist when no regularity and growth assumptions are made on $\W$ and $\dom\W(x,\cdot)$. In our work we study the existence of integral representation of $\F$ on $\dom F:=\{u\in W^{1,p}(\Omega;\RR^m): F(u)\sinf\infty\}$ without using of approximation results, and then give some applications showing how to obtain a full integral representation. Following this way, we try to establish conditions for the existence of integral representation of $\F$ with the restrictions that $\dom\W(x,\cdot)$ is convex for all $x\in\Omega$, $\W$ is $p$-coercive and $p\in ]d,\infty]$. This simplified framework allows us to deal with functions of $W^{1,p}(\Omega;\RR^m)$ that are almost everywhere differentiable in $\Omega$, which is an important ingredient for the possibility of integral representation of $\F$. The techniques we use are based on measure arguments and localization.

\section{Main results}\label{Main results}
We denote by $\mathcal O(\Omega)$ the set of all open subsets of $\Omega$. For each $O\in\O(\Omega)$, we will denote by ${W}^{1,p}_0(O;\RR^m)$ the subset of all $\phi\in W^{1,p}(\Omega;\RR^m)$ such that $\phi=0$ in $\Omega\ssetminus O$ (this definition is equivalent to the classical definition of ${W}^{1,p}_0(O;\RR^m)$ (see for instance \cite[Chap. 9, p. 233]{hedberg-adams})). We denote by $\Q$ any open cube of $\RR^d$.

Let $L:\Omega\times \MM^{m\times d}\to [0,\infty]$ be a Borel measurable integrand. We consider the following assertions:

\begin{hyp}
\item\label{H0} if $p\in ]d,\infty[$ then there exists $c\ssup0$ such that for every $(x,\xi)\in\Omega\times\MM^{m\times d}$
\[
c\vert\xi\vert^p\le L(x,\xi);
\]
\item\label{H0prime} if $p=\infty$ then there exists $R_0\ssup0$ such that 
\[
\dom L(x,\cdot)\subset \overline{Q}_{R_0}(0)\mbox{ a.e. in }\Omega;
\]
\item\label{H1} there exists $\rho_0\ssup0$ such that 
\[
\overline{Q}_{\rho_0}(0)\subset\Lambda_L:=\left\{\xi\in\MM^{m\times d}:\int_{\Omega} L(x,\xi)dx\sinf\infty\right\};
\]
\item\label{H4} for almost all $x\in\Omega$
\[
\dom L(x,\cdot)\subset\Lambda_L(x):=\left\{\xi\in\MM^{m\times d}:L(x,\xi)=\lim_{\eps\to 0}\displaystyle\fint_{\Q_\eps(x)} L(y,\xi)dy\right\};
\]
\item\label{H3} there exists $C\ssup0$ such that for every $\xi,\zeta\in\MM^{m\times d},\; x\in \Omega, \mbox{ and }t\in ]0,1[$
\[
L(x,t\xi+(1-t)\zeta)\le C\left(1+L(x,\xi)+L(x,\zeta)\right);
\]
\item\label{regintegrand} for almost all $x\in\Omega$
\begin{align*}
&\dom L(x;\cdot)\subset\Xi_L:=\left\{\xi\in\MM^{m\times d}: \limsup_{\delta\to 0}\omega_\delta^L(\xi)\sinf\infty\right\},\\
\mbox{where }\;\;
&\omega_\delta^L(\xi):=\sup_{\substack{\Q\subset\Omega\\ \diam(\Q)<\delta}}\inf_{\varphi\in W^{1,p}_0(\Q;\RR^m)}\fint_\Q L(x,\xi+\nabla \varphi(x))dx.
\end{align*}
\end{hyp}

\begin{remark} Some remarks on the previous assertions are in order:

(i) The assertion \ref{H0} (resp. \ref{H0prime}) is a coercivity condition in the case $p$ finite (resp. $p$ non finite), it is used only in the Subsection~\ref{subsect1}. Note that if $p\in ]d,\infty]$ and \ref{H0} or  \ref{H0prime} hold, then due to compact embeddings of $W^{1,p}(\Omega;\RR^m)$ in $L^{\infty}(\Omega;\RR^m)$ we have
\begin{align}\label{compact imbedding}
\F(u)=\inf\left\{\liminf_{n\to \infty} F(u_n):W^{1,p}(\Omega;\RR^m)\ni u_n\stackrel{L^{\!^\infty}}{\to} u\right\}.
\end{align}

(ii) It is easy to see that the combination (\ref{H1}, \ref{H4} and \ref{regintegrand}) is equivalent to: 
\begin{hyp}[start=7]
\item\label{H14} there exists $\rho_0\ssup0$ such that 
\[
\overline{Q}_{\rho_0}(0)\subset \Lambda_L\subset\dom L(x,\cdot)\subset\Lambda_L(x)\cap \Xi_L\;\mbox{ a.e. in }\Omega.
\]
\end{hyp}

(iii) Due to \ref{H3}, the effective domain $\dom L(x,\cdot)$ is convex for all $x\in\Omega$, the same holds for $\Lambda_L$ and $\Lambda_L(x)$. 

(iv) The assertion \ref{H1} is equivalent to $0\in\inte(\Lambda_L)$ where $\inte(\Lambda_L)$ denotes the interior of $\Lambda_L$.

\end{remark}

We denote by $Y$ the cube $]0,1[^d$. Define $\Z L:\Omega\times\MM^{m\times d}\to[0,\infty]$ by 
\begin{align*}
\Z L(x,\xi):=\liminf_{\eps\to 0}\inf\left\{\int_{Y} L(x+\eps y,\xi+\nabla\varphi(y))dy: \varphi\in {W}^{1,p}_0(Y;\RR^m)\right\}.
\end{align*}

\begin{remark} (i) The formula which gives $\Z L$ can be rewritten 
 \begin{align}\label{gendacoro}
\Z L(x,\xi):=\liminf_{\eps\to 0}\inf\left\{\fint_{\Q_\eps(x)} L(y,\xi+\nabla\varphi(y))dy: \varphi\in {W}^{1,p}_0(\Q_\eps(x);\RR^m)\right\},
\end{align}
where $\Q_\eps(x)=x+\eps Y$ with $\eps\ssup 0$ and $x\in\Omega$.

(ii) If $L$ does not depend on $x$, then $L$ is $W^{1,p}$-quasiconvex in the sense of Ball \& Murat \cite{ball-murat84} if and only if $L=\Z L$. In fact $\Z L$ is the generalization to $x$-dependent integrand of the Dacorogna quasiconfexication formula. If $L$ is a Carath\'eodory integrand with $p$-polynomial growth then we can freeze the variable $x$ and show that
\begin{align*}
\Z L(x,\xi)=\inf\left\{\int_{Y} L(x,\xi+\nabla\varphi(y))dy: \varphi\in {W}^{1,\infty}_0(Y;\RR^m)\right\}
\end{align*}
which is the Dacorogna quasiconvexification formula for each $x$ fixed. However the formula \eqref{gendacoro} can be considered as a natural generalization when we deal with Borel measurable integrand which can take the value $\infty$.
\end{remark}

\begin{definition}
We say that $L$ is {\em $W^{1,p}$-quasiconvex} if $L=\Z L$.
\end{definition}

\medskip

We say that $L$ is {\em radially uniformly upper semicontinuous (ru-usc)} if there exists $a\in L^1_{\rm loc}(\Omega;]0,\infty])$ such that $\limsup_{t\to 1}\Delta^a_L(t)\leq 0$ where $\Delta^a_L:[0,1]\to ]-\infty,\infty]$ is defined by
\begin{align*}
\Delta^a_L(t):=\esssup_{x\in\Omega}\sup_{\xi\in {\rm dom}L(x,\cdot)}\frac{L(x,t\xi)-L(x,\xi)}{a(x)+L(x,\xi)}.
\end{align*}
The systematic use of the concept of ru-usc functions in the setting of the relaxation of nonconvex functional in the vectorial case starts in \cite{oah10}, then it is used to prove homogenization results in \cite{oah-jpm11,oah-jpm12}.

Define $\widehat{\Z L}:\Omega\times\MM^{m\times d}\to[0,\infty]$ by 
\begin{align*}
\widehat{\Z L}(x,\xi)&:=\liminf_{t\to 1^-}\Z L(x,t\xi).
\end{align*}
\begin{remark} (i) In fact, if $L$ is ru-usc then $\Z L$ too (see Lemma~\ref{ZLruusc}). 

(ii) If \ref{H3} holds and $\Z L$ is ru-usc then the $\liminf$ in the definition of $\widehat{\Z L}$ is a limit (see Lemma~\ref{ruuscLimit}).  
\end{remark}

%%%%%%%%%%%%%%%
We state the main result of the paper.
\begin{theorem}\label{main result-ruusc1}
Assume that $\W$ satisfies \ref{H0} {\rm (\ref{H0prime} if $p=\infty$)}, \ref{H14} and \ref{H3}. If either $\W$ is ru-usc or $\Z\W$ is both ru-usc and $W^{1,p}$-quasiconvex, then for every $u\in\dom F$ we have
\begin{align}\label{intregral-representation}
\F(u)=\int_\Omega \widehat{\Z\W}(x,\nabla u(x))dx.
\end{align}
\end{theorem}
\begin{remark} (i) Under the same assumptions the local version of Theorem~\ref{main result-ruusc1} also holds, i.e., if we set 
\begin{align*}
\F(u;O):=\inf\left\{\liminf_{n\to \infty} \int_O \W(x,\nabla u_n(x))dx:W^{1,p}(\Omega;\RR^m)\ni u_n\wto u\right\}
\end{align*}
then 
\begin{align*}
\F(u;O)=\int_O \widehat{\Z\W}(x,\nabla u(x))dx
\end{align*}
for all $O\in\O(\Omega)$ and $u\in\dom F(\cdot;O)$ where 
\[
\dom F(\cdot;O)=\left\{u\in W^{1,p}(\Omega;\RR^m): \int_O \W(x,\nabla u(x))dx\sinf\infty\right\}.
\]
(ii) We do not know whether $\Z\W$ is $W^{1,p}$-quasiconvex (i.e., $\Z(\Z\W)=\Z\W$) when $\W$ is assumed to be ru-usc.
\end{remark}

If we consider a stronger assumption (see \ref{H2}) in place of \ref{H1} then the following result shows that the full integral representation of $\F$ holds. 
%%%%%%%%%%%%%%%%%%%%%%%%%
\begin{theorem}\label{main result-ruusc2} Assume that $\W$ satisfies \ref{H0} {\rm(\ref{H0prime} if $p=\infty$)}, \ref{H3}, \ref{H4}, \ref{regintegrand} and 
\begin{hyp}[start=8]
\item\label{H2} there exists $\rho_0\ssup 0$ such that for every $u\in W^{1,p}(\Omega;\RR^m)$
\[
\vert u\vert_{1,p}\le \rho_0\Longrightarrow \int_\Omega \W(x,\nabla u(x))dx\sinf\infty.
\]
\end{hyp}
If either $\W$ is ru-usc or $\Z\W$ is both ru-usc and $W^{1,p}$-quasiconvex then \eqref{intregral-representation} holds for all $u\in\dom \F$.
\end{theorem}

\begin{remark} (i) Under the same assumptions the local version of Theorem~\ref{main result-ruusc2} also holds, i.e., 
\begin{align*}
\F(u;O)=\int_O \widehat{\Z\W}(x,\nabla u(x))dx
\end{align*}
for all $O\in\O(\Omega)$ and $u\in\dom \F(\cdot;O)$.

(ii) Theorem~\ref{main result-ruusc2} is mainly used to propose an alternative of the results of  \cite{dearcangelis-monsurro-zappale04,dearcangelis-zappale05} (see Subsection~\ref{scalarcase}). 

(iii) Note that the assertion \ref{H2} implies \ref{H1}. It seems that condition \ref{H2} makes sense when $p=\infty$ because in this case we can show that  \ref{H1} and \ref{H3} imply \ref{H2} (see Corollary~\ref{lipschitz-scalar}). The condition \ref{H2} means that the effective domain has to be ``thick" enough in order to have no gap appears when passing from the representation on $\dom F$ to the representation on $\dom \F$.
\end{remark}

Theorems~\ref{main result-ruusc1} and~\ref{main result-ruusc2} are consequences of the following proposition. Define $\Z F:{W}^{1,p}(\Omega;\RR^m)\times \O(\Omega)\to [0,\infty]$ by
\begin{align*}
\Z F(u;O):=\inf\left\{\liminf_{n\to \infty}\int_O\Z\W(x,\nabla u_n(x))dx:{W}^{1,p}(\Omega;\RR^m)\ni u_n\wto u\right\}.
\end{align*}

\begin{proposition}\label{main proposition} 
Assume that $\W$ satisfies \ref{H0} {\rm(\ref{H0prime} if $p=\infty$)}, \ref{H1}, \ref{H4} and \ref{H3}. Let $O\in\O(\Omega)$.
\begin{enumerate}
\item[(i)] Then for every $u\in{W}^{1,p}(\Omega;\RR^m)$ such that $F(tu)\sinf\infty$ for all $t\in ]0,1[$, we have 
\begin{align*}
\F(tu;O)\le\int_O \Z \W(x,t\nabla u(x))dx
\end{align*}
for all $t\in ]0,1[$.
\item[(ii)] If $\Z\W$ is ru-usc and $W^{1,p}$-quasiconvex then for every $u\in \dom F$ we have 
\begin{align*}
\overline{\Z F}(u;O)\ge\int_O\widehat{\Z\W }(x,\nabla u(x))dx.
\end{align*}
\item[(iii)] If $\W$ is ru-usc then for every $u\in \dom F$ we have 
\begin{align*}
\F(u;O)\ge\int_O\widehat{\Z\W }(x,\nabla u(x))dx.
\end{align*}
\end{enumerate}
\end{proposition}
In fact the $\liminf$ in the definition of $\Z\W$ is a limit if \ref{H4} and \ref{regintegrand} hold.
\begin{proposition}\label{exist-ZL} If $L:\Omega\times\MM^{m\times d}$ is a Borel measurable integrand satisfying \ref{H4} and \ref{regintegrand} then for almost all $x\in\Omega$ and for every $\xi\in \dom L(x,\cdot)$
\[
\Z L(x,\xi)=\lim_{\eps\to 0}\inf\left\{\int_Y L(x+\eps y,\xi+\nabla \varphi(y)dy:\varphi\in W^{1,p}_0(Y;\RR^m)\right\}.
\]
\end{proposition}

The plan of paper is as follow. In Sect.~\ref{Some consequences} we give some applications in the case where $\W$ satisfies quasiconvex growth, we show that a full integral representation holds if the functional associated to the quasiconvex growth is sequentially weakly lsc on $W^{1,p}(\Omega;\RR^m)$. The scalar case is treated by using Theorem~\ref{main result-ruusc2} and adding some assumptions on the regularity of $\dom \W(x,\cdot)$. Finally an application of Theorem~\ref{main result-ruusc1} is developed in the context of relaxation with integrand satisfying $\p(x)$-growth.

In Sect.~\ref{Preliminary results} we first establish some results on $L$ and the envelope $\Z L$ needed for the proof of Proposition~\ref{main proposition}. Then we introduce the concept of ru-usc functionals and state abstract results needed in the proof of Theorem~\ref{main result-ruusc2}.   

In Sect.~\ref{Proof of Theorem 1} the proofs of Theorem~\ref{main result-ruusc1} and Theorem~\ref{main result-ruusc2} are given by using Proposition~\ref{main proposition}. The proof of Theorem~\ref{main result-ruusc2} use the abstract result on ru-usc functionals of Subsection~\ref{Ru-usc functionals} and especially Corollary~\ref{Extension-envelope}.

The Sect.~\ref{Proof of Proposition (i)} and Sect.~\ref{Proof of Proposition (iii)} are devoted to the proof of Proposition~\ref{main proposition}. The strategy to prove the upper bound part (i) of Proposition~\ref{main proposition} is inspired by the paper of \cite{bouchitte-fonseca-mascarenhas98}. They develop a new method to prove integral representations for relaxed variational functionals and $\Gamma$-limit of variational functionals. Roughly, their method consists, when  $\F(u;\cdot)$ is a Radon measure absolutely continuous with respect to a fixed finite nonnegative Radon measure, to express the Radon-Nikodym derivative of $\F(u;\cdot)$ in terms of minima of local Dirichlet problems for $\F$. However, in our case, we use an indirect method for the proof of Proposition~\ref{main proposition} (i) in the sense that we do not prove directly that $\F(u;\cdot)$ is a Radon measure. Note that some similar ideas considering the link between the relaxed integrand and minima of local Dirichlet problems appear in \cite{dalmaso-modica86} and in the context of $G$-convergence in \cite{degiorgi-spagnolo73}. For the proof of the lower bound parts (ii) and (iii) of Proposition~\ref{main proposition}, we use the techniques of localization (also known as blow-up method) and cut-off method introduced by \cite{fonseca-muller92}.

In Sect.~\ref{Proof of Theorem exist-ZL} we give the proof of Proposition~\ref{exist-ZL} by using some measure arguments. More precisely, the proof consists to see $\Z L(\cdot,\xi)$ as derivate of a set function (see for instance \cite{hartnett-kruse, bongiorno}).
%%%%%%%%%%%\section{Some consequences}%%%%%%%%%%%
\section{Applications}\label{Some consequences}

\subsection{Relaxation with quasiconvex growth}\label{quasiconvexgrowth}
Let $p\in ]d,\infty[$. Let $G:\MM^{m\times d}\to[0,\infty]$ be a Borel measurable integrand. Consider the assertions:
\begin{hypg}
\item\label{ass1-G} G is $W^{1,p}$-quasiconvex, i.e., for every $\xi\in\MM^{m\times d}$ 
\[
G(\xi)=\Z G(\xi).
\]
\item\label{ass2-G} $\W$ has $G$-growth, i.e., there exist $\alpha,\beta\ssup 0$ such that for every $(x,\xi)\in\Omega\times\MM^{m\times d}$
\[
\alpha G(\xi)\le \W(x,\xi)\le \beta(1+G(\xi)).
\]
\end{hypg} 

\begin{theorem}\label{main1} Assume that $G$ satisfies \ref{H0}, \ref{H3}, \ref{ass1-G}, \ref{ass2-G} and $0\in\inte(\dom G)$.  If either $\W$ is ru-usc or $\Z\W$ is both ru-usc and $W^{1,p}$-quasiconvex then  \eqref{intregral-representation} holds for all $u\in W^{1,p}(\Omega;\RR^m)$ such that $\int_\Omega G(\nabla u(x))dx\sinf\infty$. Moreover \eqref{intregral-representation} holds for all $u\in\dom\F$ if \[W^{1,p}(\Omega;\RR^m)\ni u\mapsto \int_\Omega G(\nabla u(x))dx\] is sequentially weakly lower semicontinuous (swlsc) on $W^{1,p}(\Omega;\RR^m)$.
\end{theorem}
\begin{proof} By \eqref{ass2-G}, it is easy to see that $\dom\W(x,\cdot)=\dom G=\Lambda_\W=\Lambda_\W(x)=\Xi_\W$ a.e. in $\Omega$, so $\W$ satisfies \ref{H14} since $0\in\inte(\dom G)$. We have also that $\W$ satisfies \ref{H0} since $G$ satisfies \ref{H0}. By Lemma~\ref{L growth}, $G$ satisfies \ref{H3} if and only if $\W$ satisfies \ref{H3}. Applying Theorem~\ref{main result-ruusc1} we obtain \eqref{intregral-representation} for all $u\in\dom F$. But $\dom F=\{u\in W^{1,p}(\Omega;\RR^m):\int_\Omega G(\nabla u)dx\sinf\infty\}$ since \eqref{ass2-G}. If we assume that $W^{1,p}(\Omega;\RR^m)\ni u\mapsto \int_\Omega G(\nabla u(x))dx$ is swlsc on $W^{1,p}(\Omega;\RR^m)$, then again by using \eqref{ass2-G}, $\dom \F=\{u\in W^{1,p}(\Omega;\RR^m):\int_\Omega G(\nabla u)dx\sinf\infty\}$, thus $\dom F=\dom \F$ and the integral representation holds for all $u\in\dom \F$.
\end{proof}

\subsection{Relaxation in the scalar case}\label{scalarcase}
Let $L:\Omega\times \MM^{m\times d}\to [0,\infty]$ be an integrand. We denote by $L^{\ast\ast}:\Omega\times\MM^{m\times d}\to [0,\infty]$ the convex lower semicontinuous envelope of $L(x,\cdot)$ for each $x\in\Omega$, i.e., 
\begin{align*}
L^{\ast\ast}(x,\xi):=\sup\big\{g(x,\xi): g(x,\cdot)\mbox{ is convex and lsc, }\;\; g(x,\cdot)\le L(x,\cdot)\big\}
\end{align*}
for all $(x,\xi)\in\Omega\times\MM^{m\times d}$.

\medskip

To show that the relaxed integrand $\widehat{\Z L}$ coincides with $L^{\ast\ast}$ when $m=1$ we need assumption ``\`a la De Arcangelis and all" (see \cite{dearcangelis-monsurro-zappale04,dearcangelis-zappale05}). For each $\eps\ssup 0$ define the multifunction $D_\eps:\Omega \dto\MM^{m\times d}$ by 
\begin{align*}%\label{essentielim}
D_\eps(x):=\bigcup_{\varphi\in W^{1,p}_0(\Q_\eps(x);\RR^m)}\bigcup_{\substack{N\subset \Q_\eps(x)\\ \vert N\vert=0}}\quad\bigcap_{y\in \Q_\eps(x)\setminus N}{\inte(\dom L(y,\cdot))-\{\nabla \varphi(y)\}}.
\end{align*}
Consider the following assertion:
\begin{hypc}
\item\label{essentielim} for almost all $x\in\Omega$
\[
\bigcup_{\delta> 0}\bigcap_{\eps\in ]0,\delta[}D_\eps(x)\subset{\dom L(x,\cdot)}.
\]
\end{hypc}

\begin{lemma}\label{convpre} If $L:\Omega\times \MM^{m\times d}\to [0,\infty]$ is a Borel measurable satisfying \ref{essentielim}, \ref{H1}, \ref{H4} and \ref{regintegrand} then for a.a. $x\in\Omega$ and for every $t\in ]0,1[$ we have
\begin{align*}
t\dom \widehat{\Z L}(x,\cdot)\subset \dom L(x,\cdot).
\end{align*}
\end{lemma}
\begin{proof} Fix $x\in\Omega$ such that \ref{essentielim} holds and $0\in \inte(\dom L(x,\cdot))$. Fix $\xi\in \dom \Z L(x,\cdot)$. Then by Proposition~\ref{exist-ZL} there exists $\eps_0\ssup 0$ such that 
\begin{align*}
\sup_{\eps\in ]0,\eps_0[}\fint_{\Q_\eps(x)}L(y,\xi+\nabla \varphi_\eps(y))dy\sinf\infty,
\end{align*}
for some $\{\varphi_\eps\}_{\eps\in ]0,\eps_0[},\varphi_\eps\in W^{1,p}_0(\Q_\eps(x);\RR^m)$. It follows that for every $\eps\in ]0,\eps_0[$ there exists a negligible set $N_x^\eps\subset\Q_\eps(x)$ such that for every $y\in \Q_\eps(x)\ssetminus N_x^\eps$ we have $\xi+\nabla\varphi_\eps(y)\in \dom L(y,\cdot)$. It holds $t\xi+t\nabla\varphi_\eps(y)\in t\dom L(y,\cdot)$ for all $y\in \Q_\eps(x)\ssetminus N_x^\eps$ and all $t\in]0,1[$. Hence $t\xi\in \cap_{y\in \Q_\eps(x)\ssetminus N_x^\eps}t\dom L(y,\cdot)-\{t\nabla\varphi_\eps(y)\}$ for all $t\in]0,1[$. By convexity of $\dom L(y,\cdot)$ and the fact that by \ref{H1}  we have $0\in\inte(\dom L(y,\cdot))$ for all $y\in\Q_\eps(x)\ssetminus{N^\prime}$ for some negligible set $N^\prime$, we deduce $t\dom L(y,\cdot)\subset \inte(\dom L(y,\cdot))$ for all $y\in\Q_\eps(x)\ssetminus{N^\prime}$ for all $t\in]0,1[$. It follows that for every $t\in]0,1[$
\[
t\xi\in \bigcap_{y\in \Q_\eps(x)\ssetminus \left(N_x^\eps\cup N^\prime\right)}\inte(\dom L(y,\cdot))-\{t\nabla\varphi_\eps(y)\}.
\]
From \ref{essentielim} we deduce that $t\xi\in{\dom L(x,\cdot)}$ for all $t\in]0,1[$ which completes the proof.
\end{proof}
\begin{lemma}\label{convexityZL} If the assumptions of Theorem~\ref{main result-ruusc1} and \ref{essentielim} hold then for a.a. $x\in\Omega$ the integrand $\widehat{\Z L}(x,\cdot)$ is rank-one convex and equals to $\overline{\Z L}(x,\cdot)$ the lsc envelope of $\Z L(x,\cdot)$.
\end{lemma}
\begin{proof} We have to show that for a.a. $x\in\Omega$, for every $\xi,\zeta\in\MM^{m\times d}$ such that ${\rm rank}(\xi-\zeta)\le 1$ and for every $\tau\in ]0,1[$ we have
\begin{align*}
\widehat{\Z L}(x,\tau \xi+(1-\tau)\zeta)\le \tau \widehat{\Z L}(x,\xi)+(1-\tau)\widehat{\Z L}(x,\zeta).
\end{align*}
Fix $x_0\in \Omega^\prime$ where 
\begin{align*}
 \Omega^\prime:=\left\{x\in\Omega:\forall t\in]0,1[\quad t\dom \widehat{\Z L}(x,\cdot)\subset \dom L(x,\cdot)\subset \Lambda_L(x)\right\}.
\end{align*}
Since Lemma~\ref{convpre} and \ref{H4} we have $\vert \Omega\ssetminus\Omega^\prime\vert=0$.

Fix $\xi,\zeta\in \dom \widehat{\Z L}(x_0,\cdot)$. Thus $L(x_0,t\xi)\sinf\infty$ and $L(x_0,t\zeta)\sinf\infty$ for all $t\in]0,1[$. Fix $t\in ]0,1[$. If $\chi\in \{\xi,\zeta\}$ then
\begin{align*}
\infty\ssup L(x_0,t\chi)=\lim_{\eps\to 0}\fint_{\Q_\eps(x_0)}L(y,t\chi)dy=\lim_{\eps\to 0}\frac{F(tl_\chi;\Q_\eps(x_0))}{\eps^d}
\end{align*} 
where $l_\chi(y):=\chi y$ for all $y\in\RR^d$.
Choose $\delta_0^t\ssup 0$ such that ${F(tl_\chi;\Q_{\delta_0^t}(x_0))}\sinf\infty$. By the local version of Theorem~\ref{main result-ruusc1} (with $O=\Q_\delta(x_0)$) we have for every $\eps\in ]0,\delta_0^t[$
\begin{align*}
{\F(tl_\chi;\Q_\eps(x_0))}=\int_{\Q_\eps(x_0)}\widehat{\Z L}(y,t\chi)dy.
\end{align*}
Reasoning as in the proof of the zig-zag lemma (see for instance \cite[p. 79-80]{braides98}), we obtain for every $\tau\in ]0,1[$
\[
\widehat{\Z L}(x_0,\tau t\xi+(1-\tau)t\zeta)\le \tau \widehat{\Z L}(x_0,t\xi)+(1-\tau)\widehat{\Z L}(x_0,t\zeta).
\]
Letting $t\to 1$ and using Lemma~\ref{ruuscLimit} we obtain that $\widehat{\Z L}(x_0,\cdot)$ is rank-one convex. Then $\widehat{\Z L}(x_0,\cdot)$ is separately convex so is continuous in $\inte(\dom \widehat{\Z L}(x_0,\cdot))$ (see \cite[Theorem 2.31, p. 47]{dacorogna08}). Applying Lemma~\ref{ZLruusc} together with Theorem~\ref{Extension-Result-for-ru-usc-Functions} we obtain that the function $\widehat{\Z L}(x_0,\cdot)$ is the lsc envelope of ${\Z L}(x_0,\cdot)$.
\end{proof}
The integral representation of $\F$ was studied in \cite{dearcangelis-monsurro-zappale04,dearcangelis-zappale05} in the scalar case with $p=\infty$, we propose here the following alternative result.
\begin{theorem}\label{convexification} Assume that $m=1$. Assume that $\W$ satisfies \ref{H0} {\rm(\ref{H0prime} if $p=\infty$)}, \ref{H4}, \ref{H3}, \ref{regintegrand}, \ref{H2}, and \ref{essentielim}. If either $\W$ is ru-usc or $\Z\W$ is both ru-usc and $W^{1,p}$-quasiconvex then for every $u\in\dom \F$ we have
\begin{align*}%\label{intregral-representation-scalar}
\F(u)=\int_\Omega \W^{\ast\ast}(x,\nabla u(x))dx.
\end{align*}
Moreover $\W^{\ast\ast}(x,\cdot)=\overline{\Z \W}(x,\cdot)=\widehat{\Z \W}(x,\cdot)$ a.e. in $\Omega$.
\end{theorem}
\begin{proof} By Theorem~\ref{main result-ruusc2}, the representation \eqref{intregral-representation} holds for all $u\in\dom \F$. Fix $\xi\in\MM^{1\times d}$. On one hand, by a well known lower semicontinuity result (see for instance \cite[Theorem 4.1.1]{buttazzo89}) we have for every $O\in\O(\Omega)$ and $u\in W^{1,p}(\Omega)$
\begin{align*}
\F(u;O)\ge& \inf\left\{\liminf_{n\to\infty}\int_O \W^{\ast\ast}(x,\nabla u_n(x))dx:W^{1,p}(\Omega)\ni u_n\wto u\right\}\\\ge& \int_O \W^{\ast\ast}(x,\nabla u(x))dx.
\end{align*} 
It follows that $\widehat{\Z\W}(x,\cdot)\ge\W^{\ast\ast}(x,\cdot)\mbox{ a.e. in }\Omega$. On the other hand by Lemma~\ref{convexityZL} we have $\overline{\Z\W}(x,\cdot)$ is convex and lsc and $\W(x,\cdot)\ge\overline{\W}(x,\cdot)\ge\overline{\Z\W}(x,\cdot)=\widehat{\Z \W}(x,\cdot)\ge \W^{\ast\ast}(x,\cdot)$ a.e. $\Omega$, and the proof is complete.
\end{proof}
\begin{remark} If we consider the case $p=\infty$ in Theorem~\ref{convexification}, we can replace the assumption \ref{H2} by \ref{H1} since \ref{H1} and \ref{H3} imply \ref{H2} (see Corollary~\ref{lipschitz-scalar}).
\end{remark}
%%%%%%%%%%%%%%%%%%%%%Relaxation with $\p(x)$-growth%%%%%%%%%%%%%%%
\subsection{Relaxation with $\p(x)$-growth}
Let $p\in]d,\infty[$. Let $\p:\Omega\to [0,\infty[$ be a measurable function such that $p\le \p(x)$ for all $x\in\Omega$.

Let $\W:\Omega\times\MM^{m\times d}\to [0,\infty[$ be a Borel measurable integrand. 

Consider the assertions: 
\begin{hypgr}
\item\label{Lambda-sufficiency} for each $\xi \in\MM^{m\times d}$ we have
\[
\vert\xi\vert^{\p(\cdot)}\in L^1(\Omega)\quad\mbox{ and }\quad\limsup_{\delta\to 0}\sup_{\substack{\Q\subset\Omega,\;\tiny{\rm cube}\\\diam(\Q)<\delta}}\fint_\Q \vert \xi\vert^{\p(x)}dx\sinf\infty;
\]
\item\label{pxgrowth} there exist $\alpha,\beta\ssup 0$ such that for every $(x,\xi)\in\Omega\times\MM^{m\times d}$ we have
\[
\alpha\vert \xi\vert^{\p(x)}\le \W(x,\xi)\le\beta(1+\vert\xi\vert^{\p(x)}).
\]
\end{hypgr}
When \ref{pxgrowth} holds, we say that $\W$ has $\p(x)$-growth. The condition \ref{Lambda-sufficiency} is satisfied if $\p(\cdot) \le p^\ast$ for some $p^\ast\in ]d,\infty[$.

\begin{theorem}  Assume that \ref{pxgrowth} holds. If \eqref{intregral-representation} holds for each $u\in\dom \F$, then $\widehat{\Z\W}$ is a Carath\'eodory integrand which is ru-usc and rank-one convex with respect to the second variable.
\end{theorem}
\begin{proof} Reasoning as in the proof of the zig-zag lemma (see for instance \cite[p. 79-80]{braides98}), we obtain that $\widehat{\Z\W}(x,\cdot)$ is rank-one convex for all $x\in\Omega$. Then $\widehat{\Z\W}(x,\cdot)$ is separately convex. Moreover using \ref{Lambda-sufficiency}, it is easy to see that $\widehat{\Z \W}$ satisfies: for a.a. $x\in\Omega$ and for every $\xi\in\MM^{m\times d}$ it holds
\begin{equation}\label{grzf}
\alpha\vert \xi\vert^{\p(x)}\le \widehat{\Z \W}(x,\xi)\le\beta(1+\vert\xi\vert^{\p(x)}).
\end{equation}
So by using \cite[Theorem 2.31, p. 47]{dacorogna08} we obtain that for a.a. $x\in\Omega$ the function $\widehat{\Z\W}(x,\cdot)$ is continuous in $\inte(\dom \widehat{\Z L}(x,\cdot))$. But for a.a. $x\in\Omega$ we have $\inte(\dom \widehat{\Z\W}(x,\cdot))=\dom \widehat{\Z\W}(x,\cdot)=\MM^{m\times d}$ since \eqref{grzf}.

By \eqref{grzf} and \cite[Prop. 2.32, p. 51] {dacorogna08} there exists $K\ssup 0$ such that for a.a. $x\in\Omega$, for every $t\in ]0,1[$ and every $\xi\in\MM^{m\times d}$
\begin{align*}
\vert\widehat{\Z\W}(x,t\xi)-\widehat{\Z\W}(x,\xi)\vert&\le K\vert t\xi-\xi\vert \left(1+\vert t\xi\vert^{\p(x)-1}+\vert \xi\vert^{\p(x)-1}\right)\\
&\le (1-t)4K\left(1+\vert \xi\vert^{\p(x)}\right)\\
&\le (1-t)4K\left(1+\mbox{$\frac{1}{\alpha}$}\widehat{\Z\W}(x,\xi)\right)\\
&\le (1-t)4K(1+\mbox{$\frac{1}{\alpha}$})\left(1+ \widehat{\Z\W}(x,\xi)\right)
\end{align*}
where we used \ref{pxgrowth}. We obtain $\Delta_{\widehat{\Z\W}}^1(t)\le (1-t)4K(1+\mbox{$\frac{1}{\alpha}$})$ which shows that $\widehat{\Z\W}$ is ru-usc by letting $t\to 1$.
\end{proof}
\begin{theorem}\label{pxtheorem} Assume that \ref{Lambda-sufficiency} and \ref{pxgrowth} hold. If ${\Z\W}$ is $W^{1,p}$-quasiconvex and ru-usc, then \eqref{intregral-representation} holds for each $u\in\dom \F$.
\end{theorem}
\begin{proof} Since \ref{pxgrowth} and \ref{Lambda-sufficiency} we have $\dom \W(x,\cdot)=\MM^{m\times d}$ for all $x\in\Omega$ and \ref{H14} holds. Moreover, \ref{H0} holds since $\p(\cdot)\ge p\ssup d$. It is easy to check that \ref{H3} holds since Lemma~\ref{L growth} and the fact that for each $x$ the function $\xi\mapsto \vert\xi\vert^{\p(x)}$ is convex. Apply Theorem~\ref{main result-ruusc1} we have \eqref{intregral-representation} for all $u\in\dom F$. Using convexity it is easy to see that the functional $W^{1,p}(\Omega;\RR^m)\ni u\mapsto \int_\Omega \vert \nabla u(x)\vert^{\p(x)}dx$ is swlsc on $W^{1,p}(\Omega;\RR^m)$. Hence 
\begin{align*}
\dom \F=\dom F=\left\{u\in W^{1,p}(\Omega;\RR^m):\int_\Omega\vert \nabla u(x)\vert^{\p(x)}dx\sinf\infty\right\},
\end{align*}
and the proof is complete. 
\end{proof}
%%%%%%%%%%%\section{Preliminary results}%%%%%%%%%%%%%
\section{Preliminary results}\label{Preliminary results}

\subsection{Some properties of $L$ and $\Z L$ }
The following lemma is an extension for nonconvex functions satisfying \ref{H1} and \ref{H3} of the classical local upper bound property for convex functions. 
\begin{lemma}\label{lemma finite integrand} Let $L:\Omega\times\MM^{m\times d}\to [0,\infty]$ be a Borel measurable integrand. If $L$ satisfies \ref{H1} and \ref{H3} then 
\begin{align*}
\int_\Omega \maxw(x) dx\sinf\infty\quad\mbox{ where }\quad\maxw(\cdot):=\sup_{\zeta\in\overline{\Q}_{\rho_0}(0)}L(\cdot,\zeta).
\end{align*}
\end{lemma}
\begin{proof} Each matrix $\xi\in\overline{\Q}_{\rho_0}(0)$ is identified to the vector
\begin{align*}
\xi=\left(\xi_{11},\cdots,\xi_{1d},\cdots,\xi_{i1},\cdots,\xi_{id},\cdots,\xi_{m1},\cdots,\xi_{md}\right).
\end{align*}
Consider the finite subset $\S:=\left\{(\xi_{11},\cdots,\xi_{md}): \xi_{ij}\in\{-\rho_0,0,\rho_0\}\right\}\subset \overline{\Q}_{\rho_0}(0)$ and define the function $L^\ast:\Omega\to [0,\infty]$ by
$
L^\ast(x):=\max_{\xi\in\S}L(x,\xi).
$
The function $L^\ast$ belongs to $L^1(\Omega)$ since \ref{H1}. Indeed, for each $x\in\Omega$ choose one $\xi_x\in \S$ such that  $L^\ast(x)=L(x,\xi_x)$, and for each $\xi\in\S$ consider the sets $M_\xi:=\{y\in\Omega:\xi_y=\xi\}$, then the finite family $\{M_\xi\}_{\xi\in\S}$ is pairwise disjoint, $\Omega=\mathop{\cup}_{\xi\in\S}M_\xi$ and 
\begin{align*}
\int_{\Omega} L^\ast(x)dx=\sum_{\xi\in\S}\int_{M_\xi}L(x,\xi_x)dx\le \sum_{\xi\in\S}\int_{\Omega}L(x,\xi)dx\sinf\infty.
\end{align*}

Fix $x\in\Omega$. Let 
$
\zeta=\left(\zeta_{11},\cdots,\zeta_{1d},\cdots,\zeta_{i1},\cdots,\zeta_{id},\cdots,\zeta_{m1},\cdots,\zeta_{md}\right)\in\S
$
with $\xi_{ij}=\zeta_{ij}$ for all $i\not=1$ and $j\not=1$. If $\xi_{11}\not=0$ then by \ref{H3} we have  
\begin{align}\label{Eq1: lemma finite integrand}
L(x,\xi)=&L\left(x,\mbox{$\frac{\vert\xi_{11}\vert}{\rho_0}$}\sgn(\xi_{11})\rho_0+\left(1-\mbox{$\frac{\vert\xi_{11}\vert}{\rho_0}$}\right)0,\cdots,\xi_{1d},\cdots,\xi_{m1},\cdots,\xi_{md}\right)\\
\le& C\left(1+L(x,\rho_0,\cdots,\xi_{md})+L(x,0,\cdots,\xi_{md})\right)\notag\\
\le& 2C\left(1+L^\ast(x)\right)\notag
\end{align}
where $\sgn(\xi_{ij})$ denotes the sign of $\xi_{ij}$. The same upper bound in \eqref{Eq1: lemma finite integrand} holds for $L(x,\xi)$ when $\xi_{11}=0$. 

Assume now that $\xi_{ij}=\zeta_{ij}$ for all $i\not=1$ and $j\notin\{1,2\}$. Then by using \eqref{Eq1: lemma finite integrand} and \ref{H3}, we have
\begin{align*}
L(x,\xi)=&L\left(x,\xi_{11},\mbox{$\frac{\vert\xi_{12}\vert}{\rho_0}$}\sgn(\xi_{12})\rho_0+\left(1-\mbox{$\frac{\vert\xi_{12}\vert}{\rho_0}$}\right)0,\cdots,\xi_{1d},\cdots,\xi_{m1},\cdots,\xi_{md}\right)\\
\le& C\left(1+2C(1+L^\ast(x))+L^\ast(x)\right).\notag\\
\le& C(1+2C)\left(1+L^\ast(x)\right).\notag
\end{align*}
Recursively, we obtain $C^\ast\ssup0$ which depends on $C$ only, such that 
\[
L(x,\xi)\le C^\ast(1+L^\ast(x))
\] 
for all $(x,\xi)\in\Omega\times\overline{\Q}_{\rho_0}(0)$. Integrating over $\Omega$ we obtain
\begin{align*}
\int_\Omega\sup_{\zeta\in\overline{\Q}_{\rho_0}(0)}L(x,\zeta)dx\le C^\ast\left(\vert\Omega\vert+\int_\Omega L^\ast(x)dx\right)\sinf\infty.
\end{align*}
\end{proof}
%%%%%%%%%%%%%
\begin{corollary}\label{lipschitz-scalar}If $p=\infty$ then \ref{H1} and \ref{H3} imply \ref{H2}, i.e., there exists $\rho_0\ssup 0$ such that $\int_\Omega L(x,\nabla u(x))dx\sinf\infty$ whenever $\vert u\vert_{1,\infty}\le \rho_0$ for all $u\in W^{1,\infty}(\Omega;\RR^m)$.
\end{corollary}
\begin{proof} Remark that $\nabla u(\cdot)\in \overline{\Q}_{\rho_0}(0)$ a.e. in $\Omega$ when $\vert u\vert_{1,\infty}\le \rho_0$, and then apply Lemma~\ref{lemma finite integrand} to complete the proof.
\end{proof}

For the proof of Theorem~\ref{main result-ruusc1} we need the following result.
\begin{lemma}\label{ZWlowerW} If $L:\Omega\times \MM^{m\times d}\to [0,\infty]$ is a Borel measurable integrand which satisfies \ref{H4} then for a.a. $x\in\Omega$ and for every $\xi\in\MM^{m\times d}$ we have \[\Z L(x,\xi)\le L(x,\xi).\]
\end{lemma}
\begin{proof} Fix $x_0\in \Omega^\prime$ where $\Omega^\prime:=\{x\in\Omega:\dom L(x,\cdot)\subset\Lambda_L(x)\}$. We have $\vert\Omega\ssetminus\Omega^\prime\vert=0$ since \ref{H4}. If $\xi\notin\dom L(x_0,\cdot)$ then $\Z L(x_0,\xi)\le \infty=L(x_0,\xi)$. Now, if $\xi\in\dom L(x_0,\cdot)$ then by \ref{H4} $\lim_{\eps\to 0}\fint_{\Q_\eps(x_0)}L(z,\xi)dz=L(x_0,\xi)$. Using the definition of $\Z L$ we finish the proof.
\end{proof}

\begin{lemma}\label{ZL satisfies H3} If $L:\Omega\times \MM^{m\times d}\to [0,\infty]$ is a Borel measurable integrand which satisfies \ref{H3} then $\Z L$ satisfies \ref{H3}.
\end{lemma}
\begin{proof} Let $x\in\Omega$, $\xi,\zeta\in\MM^{m\times d}$ and $t\in ]0,1[$. There exist $\{\varphi_{\xi,\eps}\}_\eps,\{\varphi_{\zeta,\eps}\}_\eps\subset W^{1,p}_0(\Q_\eps(x);\RR^m)$ such that
\begin{align*}
&\lim_{\eps\to 0}\fint_{\Q_\eps(x)} L(y,\xi+\nabla \varphi_{\xi,\eps})dy= \Z L(x,\xi)\\
&\lim_{\eps\to 0}\fint_{\Q_\eps(x)} L(y,\zeta+\nabla \varphi_{\zeta,\eps})dy= \Z L(x,\zeta).
\end{align*}
Since $\varphi_\eps:=t\varphi_{\xi,\eps}+(1-t)\varphi_{\zeta,\eps}\in W^{1,p}_0(\Q_\eps(x);\RR^m)$ we have 
\begin{align*}
\Z L(x,t\xi+(1-t)\zeta)&\le \liminf_{\eps\to 0}\fint_{\Q_\eps(x)} L(y,t\xi+(1-t)\zeta+\nabla \varphi_\eps)dy\\
&\le C\liminf_{\eps\to 0}\fint_{\Q_\eps(x)} \big(1+ L(y,\xi+\nabla\varphi_{\xi,\eps})+L(y,\xi+\nabla\varphi_{\zeta,\eps})\big)dy\\
&\le C\big(1+\Z L(x,\xi)+\Z L(x,\zeta)\big)
\end{align*}
which completes the proof.
\end{proof}

The following result shows that the condition \ref{H3} is shared by integrands with the same growth.
\begin{lemma}\label{L growth} If $L_1,L_2:\Omega\times\MM^{m\times d}\to [0,\infty]$ are two Borel measurable integrands such that for some $\alpha,\beta\ssup 0$ and for every $(x,\xi)\in\Omega\times \MM^{m\times d}$ it holds
\begin{align*}
\alpha L_2(x,\xi)\le L_1(x,\xi)\le\beta(1+L_2(x,\xi)),
\end{align*}
then $L_1$ satisfies \ref{H3} if and only if $L_2$ satisfies \ref{H3}.
\end{lemma}
\begin{proof} Assume that $L_1$ satisfies \ref{H3}. Fix $x\in\Omega$. Let $\xi,\zeta\in\MM^{m\times d}$ and $t\in ]0,1[$. Then
\begin{align*}
L_2(x,t\xi+(1-t)\zeta)&\le \frac1\alpha L_1(x,t\xi+(1-t)\zeta)\\
&\le\frac{C}{\alpha}\big(1+L_1(x,\xi)+L_1(x,\zeta)\big)\\
&\le \frac{C}{\alpha}\big(1+2\beta + \beta L_2(x,\xi)+\beta L_2(x,\zeta)\big)\\
&\le \frac{C}{\alpha}(1+2\beta)\big(1+L_2(x,\xi)+L_2(x,\zeta)\big).
\end{align*}
In the same manner we can verify that if $L_2$ satisfies \ref{H3} then $L_1$ too.
\end{proof}

%%%%%%%%%%%RU-USC FUNCTIONALS%%%%%%%%%%%%%%%%%%%
\subsection{Ru-usc functionals}\label{Ru-usc functionals} Let $(X,\tau)$ be a topological vector space and $J:X\to [0,\infty]$ be a function. For each $a\ssup 0$ and $D\subset 
\dom J$ we define $\Delta_{J,D}^a:[0,1]\to ]-\infty,\infty]$ by
\begin{align*}
\Delta_{J,D}^a(t):=\sup_{u\in D}\frac{J(tu)-J(u)}{a+J(u)}.
\end{align*}
When $D=\dom J$ we will write $\Delta_{J}^a:=\Delta_{J,D}^a$
\begin{definition}\label{ru-usc-Def-functional}
Given $D\subset \dom J$, we say that $J$ is {\em ru-usc in $D$}, if there exists $a\ssup 0$ such that
\begin{align*}
 \limsup_{t\to 1}\Delta^a_{J,D}(t)\leq 0.
\end{align*}
 \end{definition}
 
 \begin{remark}\label{ReMaRk-Delta-ru-usc-remARK1-functional}
 If $J$ is ru-usc in $D$ then 
 \begin{equation}\label{Delta-ru-usc-remARK1-functional}
 \limsup_{t\to1}J(tu)\leq J(u)\mbox{ for all }u\in D
 \end{equation}
Indeed, given $u\in D$, we have
\begin{align*}
 J(tu)\leq\Delta^a_{J,D}(t)\left(a+J(u)\right)+J(u)\hbox{ for all }t\in[0,1]
\end{align*}
 which gives \eqref{Delta-ru-usc-remARK1-functional} since $a+J(u)>0$ and $\limsup_{t\to 1}\Delta^a_J(t)\leq 0$.
 \end{remark}
 \begin{remark}
 If there exists $u_0\in D$ such that $J$ is ``radially" lower semicontinuous at $u_0$ in the sense that 
 \begin{align}\label{radially-lsc}
\liminf_{t\to 1}J(tu_0)-J(u_0)\ge 0.
\end{align}
 Then
 \begin{equation}\label{Delta-ru-usc-sci-functional}
 \liminf_{t\to 1}\Delta^a_{J,D}(t)\geq 0\mbox{ for all }a\ssup 0
 \end{equation}
Indeed, given such $u\in D$, for any $a\ssup 0$ we have
\begin{align*}
 \Delta_{J,D}^a(t)\geq {J(tu_0)-J(u_0)\over a+J(u_0)}\hbox{ for all }t\in[0,1]
\end{align*}
 which gives \eqref{Delta-ru-usc-sci-functional} since $a+J(u_0)>0$ and \eqref{radially-lsc}.
 \end{remark}
 
For a subset $D\subset X$, we denote by $\overline{D}^\tau$ the closure of $D$ with respect to $\tau$. 
\begin{lemma}\label{Pre-Extension-Result-for-ru-usc-Functionals} Let $D\subset \dom J$ be a $\tau$-star shaped subset with respect to $0$, i.e.,
\begin{equation}\label{Homothecie-Assumption}
t\overline{D}^\tau\subset D\hbox{ for all }t\in]0,1[.
\end{equation}
If $J$ is ru-usc in $D$ then
\begin{align*}
\liminf_{t\to 1}J(tu)=\limsup_{t\to 1}J(tu)
\end{align*}
for all $u\in \overline{D}^\tau$.
\end{lemma} 
\begin{proof}
Fix $u\in \overline{D}^\tau$. It suffices to prove that
\begin{equation}\label{ru-usc-AssUmPtIoN1}
\limsup_{t\to 1}J(tu)\leq\liminf_{t\to 1}J(tu).
\end{equation}
Without loss of generality we can assume that $\liminf_{t\to 1}J(tu)<\infty$ and there exist $\{t_n\}_n, \{s_n\}_n\subset]0,1[$ such that:
\begin{itemize}
\item $t_n\to 1$, $s_n\to 1$ and ${t_n\over s_n}\to 1$;
\item $\displaystyle\limsup_{t\to 1}J(tu)=\lim_{n\to\infty}J(t_nu)$;
\item $\displaystyle\liminf_{t\to 1}J(tu)=\lim_{n\to\infty}J(s_nu)$.
\end{itemize}
From \eqref{Homothecie-Assumption} we see that for every $n\geq 1$, $s_nu\in D$ so we can assert that for every $n\geq 1$,
\begin{equation}\label{ru-usc-AssUmPtIoN2}
J(t_n u)\leq a\Delta^a_{J,D}\left({t_n\over s_n}\right)+\left(1+\Delta^a_{J,D}\left({t_n\over s_n}\right)\right)J(s_nu).
\end{equation}
On the other hand, as $J$ is ru-usc in $D$ we have $\limsup_{n\to\infty}\left(1+\Delta^a_{J,D}\left({t_n\over s_n}\right)\right)\leq 1$ and  $\limsup_{n\to\infty}a\Delta^a_{J,D}\left({t_n\over s_n}\right)\leq 0$ since $a>0$, and \eqref{ru-usc-AssUmPtIoN1} follows from \eqref{ru-usc-AssUmPtIoN2} by letting $n\to\infty$.
\end{proof}

\medskip

Define $\widehat{J}:X\to[0,\infty]$ by 
\begin{align*}
\widehat{J}(u):=
\liminf_{t\to 1}J(tu).
\end{align*}
 
\begin{lemma}\label{hatjruusc} If $J$ is ru-usc in a $\tau$-star shaped set $D\subset\dom J$ then $\widehat{J}$ is ru-usc in $\overline{D}^\tau\cap \dom \widehat{J}$.
\end{lemma}
\begin{proof} Fix $t\in]0,1[$ and $u\in \overline{D}^\tau \cap \dom \widehat{J}$. We have $tu\in D$ since \eqref{Homothecie-Assumption} holds. By Lemma~\ref{Pre-Extension-Result-for-ru-usc-Functionals} we can assert that:
\begin{itemize}
\item $\widehat{J}(u)=\lim\limits_{s\to 1}J(su)$;
\item $\widehat{J}(tu)=\lim\limits_{s\to 1}J(s(tu))$,
\end{itemize}
and consequently
\begin{equation}\label{LimiT-Prop-ru-usc-FUNcTIONals}
{\widehat{J}(tu)-\widehat{J}(u)\over a+\widehat{J}(u)}=\lim_{s\to1}{J(t(su))-J(su)\over a+J(su)}.
\end{equation}
On the other hand, by \eqref{Homothecie-Assumption} we have $su\in D$ for all $s\in]0,1[$ so
\begin{align*}
{J(t(su))-J(su)\over a+J(su)}\leq\Delta^a_{L,D}(t)\hbox{ for all }s\in]0,1[.
\end{align*}
Letting $s\to 1$ and using \eqref{LimiT-Prop-ru-usc-FUNcTIONals} we deduce that $\Delta^a_{\widehat{J},\overline{D}^\tau\cap \dom \widehat{J}}(t)\leq\Delta^a_{J,D}(t)$ for all $t\in]0,1[$, which implies that $\widehat{J}$ is ru-usc in $\overline{D}^\tau \cap \dom \widehat{J}$ since $J$ is ru-usc in $D$.
\end{proof}

\begin{theorem}\label{Extension-Result-for-ru-usc-Functionals}
If $J$ is ru-usc in a $\tau$-star shaped set $D\subset\dom J$, and $\tau$ sequentially lower semicontinuous on $D$ then{\rm:}
\begin{itemize}
\item[(i)] $
\widehat{J}(u)=\left\{
\begin{array}{ll}
J(u)&\hbox{if }u\in D\\
\lim\limits_{t\to 1}J(tu)&\hbox{if }u\in\overline{D}^\tau\ssetminus D
\end{array}
\right.
$
\item[(ii)]$\widehat{J}=\overline{J}^{D}$ on $\overline{D}^\tau$ where 
\[\overline{J}^{D}(u):=\inf\left\{\liminf_{n\to \infty}J(u_n):D\ni u_n\stackrel{\tau}{\to}u\right\}.\]
\end{itemize}
\end{theorem}
\begin{proof}
(i) By Lemma~\ref{Pre-Extension-Result-for-ru-usc-Functionals} we have $\widehat{J}(u)=\lim_{t\to1}J(tu)$ for all $u\in \overline{D}^\tau$. From Remark~\ref{ReMaRk-Delta-ru-usc-remARK1-functional} we see that if $u\in D$ then $\limsup_{t\to1}J(tu)\leq J(u)$. On the other hand, from \eqref{Homothecie-Assumption} it follows that if $u\in D$ then $tu\in D$ for all $t\in]0,1[$. Thus, $\liminf_{t\to 1}J(tu)\geq J(u)$ whenever $u\in D$ since $J$ is $\tau$ lsc on $D$, and (i) follows. 

(ii) Let $u\in \overline{D}^\tau$. Using \eqref{Homothecie-Assumption} we have $tu\in D$ for all $t\in]0,1[$, so by Remark~\ref{ReMaRk-Delta-ru-usc-remARK1-functional} and lower semicontinuity it follows that $\widehat{J}(u)=\lim_{t\to 1}J(tu)\ge \overline{J}^D(u)$. It remains to prove that 
\begin{align}\label{(iii)-prop-ru-usc-target0-f}
\overline{J}^D(u)\ge\widehat{J}(u). 
\end{align}
Choose a sequence $\{u_n\}_n\subset D$ such that $u_n\stackrel{\tau}{\to} u$ and $\lim_{n\to \infty}J(u_n)=\overline{J}^D(u)$.
By \eqref{Homothecie-Assumption} we see that $tu_n\in D$ for all $t\in]0,1[$ and all $n\geq 1$, and consequently
\begin{align*}
\liminf_{n\to \infty}J(tu_n)\geq J(tu)\hbox{ for all }t\in]0,1[
\end{align*}
because $J$ is $\tau$ lsc on $D$. It follows that
\begin{equation}\label{(iii)-prop-ru-usc-target1-f}
\limsup_{t\to1}\liminf_{n\to \infty}J(tu_n)\geq\widehat{J}(u).
\end{equation}
On the other hand, for every $n\geq 1$ and every $t\in [0,1]$, we have
\begin{align*}
J(tu_n)\leq(1+\Delta^a_{J,D}(t))J(u_n)+a\Delta^a_{J,D}(t).
\end{align*}
As $J$ is ru-usc in $D$, letting $n\to\infty$ and $t\to1$ we obtain
\begin{align*}
\limsup_{t\to1}\liminf_{n\to \infty}J(tu_n)\leq\lim_{n\to\infty}J(u_n)=\overline{J}^D(u)
\end{align*}
which gives \eqref{(iii)-prop-ru-usc-target0-f} by combining it with \eqref{(iii)-prop-ru-usc-target1-f}.
\end{proof}
The following result is a consequence of Theorem~\ref{Extension-Result-for-ru-usc-Functionals}. For a functional $F:X\to [0,\infty]$ we denote by $\overline{F}:X\to [0,\infty]$ the $\tau$ sequential lsc envelope defined by
\begin{align*}
\overline{F}(u):=\inf\left\{\liminf_{n\to \infty}F(u_n):X\ni u_n\stackrel{\tau}{\to}u\right\}.            
\end{align*}
\begin{corollary}\label{Extension-envelope}
Assume that $\dom F$ is $\tau$-star shaped with respect to $0$, $\overline{F}$ is ru-usc in $\dom F$, and $\overline{F}=I$ on $\dom F$ where $I:X\to [0,\infty]$ is a functional. Then
\begin{align}\label{Extension-envelope-formula}
\overline{F}(u):=\left\{
\begin{array}{ll}
 I(u) &\mbox{ if }u\in\dom F     \\
\displaystyle\lim_{t\to1}  I(tu)&  \mbox{ if }u\in{\dom \overline{F}}\setminus\dom F    \\
  \infty&   \mbox{ otherwise.}  
\end{array}
\right.
\end{align}
\end{corollary}
\begin{proof} We have $\overline{F}=\overline{F}^{\dom F}=\overline{I}^{\dom F}$ since $\overline{F}=I$ on $\dom F$. Hence $I=\overline{I}^{\dom F}$ on $\dom F$ so $I$ is $\tau$ lsc on $\dom F$. Apply Theorem \ref{Extension-Result-for-ru-usc-Functionals} with $I=J$ and $D=\dom F$, it follows that $\widehat{I}=\overline{I}^{\dom F}$ so $\overline{F}=\widehat{I}$ on $\overline{\dom F}^\tau$.
\end{proof}
%%%%%%%%%%%%%%
\subsubsection{Ru-usc integrands}
Let $M\subset\RR^d$ be a measurable set and let $L:M\times\MM^{m\times d}\to[0,\infty]$ be a measurable integrand. For each $x\in M$ and for each $a\in L^1_{\rm loc}(M;]0,\infty])$, we define $\Delta_L^a:[0,1]\to]-\infty,\infty]$ by
\begin{align*}
 \Delta_L^a(t):=\esssup_{x\in M}\sup_{\xi\in \dom L(x,\cdot)}{L(x,t\xi)-L(x,\xi)\over a(x)+L(x,\xi)}.
\end{align*}
 \begin{definition}\label{ru-usc-Def}
 We say that $L$ is radially uniformly upper semicontinuous (ru-usc) if there exists $a\in L^1_{\rm loc}(M;]0,\infty])$ such that
\begin{align*}
 \limsup_{t\to 1}\Delta^a_L(t)\leq 0.
\end{align*}
 \end{definition}
 \begin{remark}\label{ReMaRk-Delta-ru-usc-remARK1}
 If $L$ is ru-usc then 
 \begin{equation}\label{Delta-ru-usc-remARK1}
 \limsup_{t\to1}L(x,t\xi)\leq L(x,\xi)
 \end{equation}
 for all $x\in M$ and all $\xi\in\dom L(x,\cdot)$. 
 \end{remark}
 \begin{remark}
 If there exist $x\in M$ and $\xi\in\dom L(x,\cdot)$ such that $L(x,\cdot)$ is lsc at $\xi$ then
 \begin{equation}\label{Delta-ru-usc-sci}
 \liminf_{t\to 1}\Delta^a_L(t)\geq 0
 \end{equation}
 for all $a\in L^1_{\rm loc}(M;]0,\infty])$.
 \end{remark}

\begin{lemma}\label{ZLruusc} If $L:\Omega\times \MM^{m\times d}\to [0,\infty]$ is a ru-usc Borel measurable integrand then $\Z L$ is ru-usc and $\Delta_{\Z L}^a(t)\le \Delta_{L}^a(t)$ for all $t\in ]0,1[$.
\end{lemma}
\begin{proof}
Indeed, fix $x\in\Omega$ a Lebesgue point for the function $a\in L^1(\Omega)$ which appears in the definition of the ru-usc for $L$. Fix $\xi\in \dom\Z L(x,\cdot)$ and choose a sequence $\{\varphi_\eps\}_\eps\subset W^{1,p}_0(\Q_\eps(x);\RR^m)$ satisfying for every $\eps\ssup 0$
\begin{align*}
\eps+\Z L(x,\xi)\ge\fint_{\Q_\eps(x)} L(y,\xi+\nabla \varphi_\eps)dy.
\end{align*}
Fix $\eps\ssup 0$. Then $\xi+\nabla\varphi_\eps(y)\in\dom L(y,\cdot)$ a.e. in $\Q_\eps(x)$. We have for every $t\in]0,1[$
\begin{align*}
&\fint_{\Q_\eps(x)} L(y,t\xi+\nabla \varphi_\eps)dy-\Z L(x,\xi)\\
&\le\fint_{\Q_\eps(x)} L(y,t\xi+\nabla \varphi_\eps)dy-\fint_{\Q_\eps(x)} L(y,\xi+\nabla \varphi_\eps)dy+\eps\\
&\le \Delta_L(t)\left(\fint_{\Q_\eps(x)}a(y)+L(y,\xi+\nabla\varphi_\eps)dy \right)+\eps\\
&\le \Delta_L(t)\left(\fint_{\Q_\eps(x)}a(y)dy+\eps+\Z L(x,\xi) \right)+\eps.
\end{align*}
Taking the infimum over all $\varphi\in W^{1,p}_0(\Q_\eps(x);\RR^m)$ and passing to the limit $\eps\to 0$ we obtain
$\Delta_{\Z L}^a(t)\le \Delta_L^a(t)$ for all $t\in ]0,1[$, and the proof is complete.
\end{proof}

\medskip

Define $\widehat{L}:M\times\MM^{m\times d}\to[0,\infty]$ by 
\begin{align*}
\widehat{L}(x,\xi):=
\liminf_{t\to 1}L(x,t\xi).
\end{align*}

The proof of the following result is similar to the proof of Lemma~\ref{Pre-Extension-Result-for-ru-usc-Functionals} and Lemma ~\ref{hatjruusc}.
\begin{lemma}\label{ruuscLimit} If $L$ is ru-usc and for a.a. $x\in M$ the effective domain $\dom L(x,\cdot)$ is star shaped with respect to $0$, i.e., $t\overline{\dom L(x,\cdot)}\subset \dom L(x,\cdot)$ for all $t\in ]0,1[$ then $\widehat L$ is ru-usc and
\[
\widehat L(x,\xi)=\lim_{t\to 1}L(x,t\xi).
\]
\end{lemma}
We can now state the analogue of Theorem~\ref{Extension-Result-for-ru-usc-Functionals}.
%%%%%%%%%%%%%
\begin{theorem}\label{Extension-Result-for-ru-usc-Functions}
Assume $L$ is ru-usc. If for every $x\in M$,
\begin{equation*}%\label{Homothecie-Assumption-Bis}
t\overline{\dom L(x,\cdot)}\subset{\rm int}(\dom L(x,\cdot))\hbox{ for all }t\in]0,1[
\end{equation*}
and $L(x,\cdot)$ is lsc on ${\rm int}(\dom L(x,\cdot))$, %where ${\rm int}(\dom L(x,\cdot))$ denotes the interior of $\dom L(x,\cdot)$, 
then{\rm:}
\begin{itemize}
\item[{(i)}] $
\widehat{L}(x,\xi)=\left\{
\begin{array}{ll}
L(x,\xi)&\hbox{if }\xi\in{\rm int}(\dom L(x,\cdot))\\
\lim\limits_{t\to 1}L(x,t\xi)&\hbox{if }\xi\in\partial\dom L(x,\cdot)\\
\infty&\hbox{otherwise{\rm;}}
\end{array}
\right.
$
\item[{(ii)}] for every $x\in M$, $\widehat{L}(x,\cdot)$ is the lsc envelope of $L(x,\cdot)$.
\end{itemize}
\end{theorem}

%%%%%%%%%%%%%%%%%Proof of Theorem~\ref{main result-ruusc1}%%%%%%%%%%%%
\section{Proof of Theorem~\ref{main result-ruusc1} and ~\ref{main result-ruusc2}}\label{Proof of Theorem 1}
\subsection{Proof of Theorem~\ref{main result-ruusc1}}
Let $O\in\O(\Omega)$ and $u\in \dom F(\cdot,O)$. If $\Z\W$ is ru-usc and $W^{1,p}$-quasiconvex, then by Lemma~\ref{ZWlowerW} we have
\begin{align*}
\F(u,O)\ge \inf\left\{\liminf_{n\to \infty}\int_O\Z\W(x,\nabla u_n(x))dx:{W}^{1,p}(\Omega;\RR^m)\ni u_n\wto u\right\}.
\end{align*}
Using Proposition~\ref{main proposition} (ii) we obtain
\begin{align}\label{gliminf}
\F(u,O)\ge \int_O \widehat{\Z\W}(x,\nabla u(x))dx.
\end{align}
If $\W$ is ru-usc then \eqref{gliminf} holds by Proposition~\ref{main proposition} (iii).

To prove the reverse inequality, note that $tu\in\dom F(\cdot,O)$ for all $t\in ]0,1[$ since \ref{H3} and \ref{H1}. Using Lemma~\ref{ZWlowerW} and \ref{H3} we have for every $t\in ]0,1[$
\begin{align*}
\Z\W(x,t\nabla u(x))\le \W(x,t\nabla u(x))\le C(1+\W(x,0)+\W(x,\nabla u(x)) \;\mbox{ a.e. in }O.
\end{align*}
Then we consider both Proposition~\ref{main proposition} (i) and Lemma~\ref{ruuscLimit} and apply the Lebesgue dominated theorem we obtain
\begin{align*}
\F(u,O)\le \liminf_{t\to 1^-}\F(tu,O)\le\limsup_{t\to 1^-}\F(tu,O)&\le\limsup_{t\to 1^-}\int_O\Z\W(x,t\nabla u)dx\\&\le\int_O \widehat{\Z\W}(x,\nabla u)dx
\end{align*}
where we used the fact that $\F(\cdot,O)$ is swlsc on $W^{1,p}(\Omega;\RR^m)$, which completes the proof.
\hfill$\blacksquare$
\subsection{Proof of Theorem~\ref{main result-ruusc2}} Fix $O\in\O(\Omega)$. By \ref{H2} and \ref{H3}, $\dom F(\cdot,O)$ is a convex subset of $W^{1,p}(\Omega;\RR^m)$ with $0$ belongs to the interior of $\dom F(\cdot,O)$ with respect to the norm topology of $W^{1,p}(\Omega;\RR^m)$. Hence, by a well-known property of convex set in normed space, we have $t\overline{\dom F(\cdot,O)}^{s}\subset \dom F(\cdot,O)$ for all $t\in [0,1[$, where $\overline{\dom F(\cdot,O)}^{s}$ is the closure of $\dom F(\cdot,O)$ in $W^{1,p}(\Omega;\RR^m)$. Since $\dom F(\cdot,O)$ is a convex set we have $\overline{\dom F(\cdot,O)}^{s}=\overline{\dom F(\cdot,O)}^{w}$ where $\overline{\dom F(\cdot,O)}^{w}$ is the closure of $\dom F(\cdot,O)$ with respect to the weak topology of $W^{1,p}(\Omega;\RR^m)$. We deduce that $\dom F(\cdot,O)$ is weakly star-shaped with respect to $0$, i.e.,
\begin{align}\label{hyp1resultruusc2}
t\overline{\dom F(\cdot,O)}^{w}\subset \dom F(\cdot,O)\mbox{ for all }t\in [0,1[.
\end{align}

We claim that $\F(\cdot,O)$ is ru-usc in $\dom F(\cdot,O)$. Indeed, let $u\in\dom F(\cdot,O)$ and $t\in ]0,1[$. First, by Proposition~\ref{main proposition} (i)
\begin{align*}
\overline{F}(tu,O)\le \int_O \widehat{\Z\W}(x,t\nabla u)dx&\le \int_O \Delta_{\widehat{\Z\W}}^a(t)\left(a(x)+\widehat{\Z\W}(x,\nabla u)\right)+\widehat{\Z\W}(x,\nabla u)dx\\&=\Delta_{\widehat{\Z\W}}^a(t)\left(\vert a\vert_{L^1(O)}+\F(u,O)\right)+\F(u,O).
\end{align*}
It follows that $\Delta_{\F(\cdot,O),\dom F(\cdot,O)}^{\vert a\vert_{L^1(O)}}(t)\le \Delta_{\widehat{\Z\W}}^a(t)$ for all $t\in ]0,1[$. By Lemma~\ref{ruuscLimit} (if $\W$ is ru-usc then combine Lemma~\ref{ZLruusc}, Lemma~\ref{ruuscLimit} and ($\dom\widehat{\Z\W}\subset\overline{\dom\Z\W}$)) $\widehat{\Z\W}$ is ru-usc, it follows that $\F(\cdot,O)$ is ru-usc in $\dom F(\cdot,O)$. Applying Corollary~\ref{Extension-envelope} with $I(u)=\int_O \widehat{\Z\W}(x,\nabla u)dx$, $D=\dom F(\cdot,O)$ and by taking account of \eqref{hyp1resultruusc2}, we obtain
\begin{align}\label{Extension-envelope-formula-2}
\overline{F}(u,O):=\left\{
\begin{array}{ll}
\displaystyle \int_O \widehat{\Z\W}(x,\nabla u)dx &\mbox{ if }u\in\dom F(\cdot,O)     \\
\displaystyle\lim_{t\to1}  \int_O \widehat{\Z\W}(x,t\nabla u)dx&  \mbox{ if }u\in{\dom \overline{F}(\cdot,O)}\setminus\dom F(\cdot,O)    \\
  \infty&   \mbox{ otherwise.}  
\end{array}
\right.
\end{align}
Let $u\in \dom \F(\cdot,O)\ssetminus\dom F(\cdot,O)$. If $I(u)=\infty$ then $\F(u,O)=\infty$, indeed, since\[\liminf_{t\to 1}\liminf_{s\to 1}\Z \W(x,st\xi)\ge \widehat{\Z\W}(x,\xi),\] we have
\begin{align*}
\F(u,O)=\lim_{t\to1}  \int_O \widehat{\Z\W}(x,t\nabla u)dx\ge \int_O \liminf_{t\to 1} \widehat{\Z\W}(x,t\nabla u)dx\ge \int_O \widehat{\Z\W}(x,\nabla u)dx=I(u)%=\infty.
\end{align*}
where we used \eqref{Extension-envelope-formula-2} and Fatou lemma. Assume now that $I(u)\sinf\infty$. On one hand, $\widehat{\Z\W}(\cdot,\nabla u(\cdot))\in L^1(O)$, and on the other hand $\widehat{\Z\W}$ is ru-usc, hence 
\begin{align*}
 \widehat{\Z\W}(x,t\nabla u(x))\le  \widehat{\Z\W}(x,\nabla u(x))+\Delta_{ \widehat{\Z\W}}^a(t)\left(a(x)+\widehat{\Z\W}(x,\nabla u(x))\right)
\end{align*}
for all $t\in ]0,1[$ and $x\in O$. Applying the Lebesgue dominated theorem we finally obtain 
\begin{align*}
\lim_{t\to1}  \int_O \widehat{\Z\W}(x,t\nabla u)dx=\int_O \widehat{\Z\W}(x,\nabla u)dx.
\end{align*}
\hfill$\blacksquare$

%%%%%%%%%%%%%%%%%%UPPER BOUND%%%%%%%%%%%%%%%%
\section{Proof of Proposition~\ref{main proposition} {\rm (i)}}\label{Proof of Proposition (i)}
\subsection{Local Dirichlet problems associated to a functional}
For any functional $H:W^{1,p}(\Omega;\RR^m)\times\O(\Omega)\to [0,\infty]$ we set
\begin{align*}
\M_H(u;O):=\inf\left\{H(v;O):v\in u+{W}^{1,p}_0(O;\RR^m)\right\}.
\end{align*}
Note that we can write $\M_H(u;O)=\inf\left\{H(u+\varphi;O):\varphi\in {W}^{1,p}_0(O;\RR^m)\right\}$ also.
For each $\eps\ssup 0$ and each $O\in\O(\Omega)$, denote by $\mathcal{V}_\eps(O)$ the class of all countable family $\{\overline{\Q}_i:=\overline{\Q}_{\rho_i}(x_i)\}_{i\in I}$ of disjointed (pairwise disjoint) closed balls of $O$ with $x_i\in O$ and $\rho_i=\diam(\Q_i)\in ]0,\eps[$ such that $\left\vert O\ssetminus \mathop{\cup}_{i\in I}Q_i\right\vert=0$. Consider $\M^\eps_H(u;\cdot):\O(\Omega)\to [0,\infty]$ given by 
\begin{align*}
\M^\eps_H(u;O):=\inf\left\{\sum_{i\in I}\M_H(u;Q_i): \{\overline{\Q}_i\}_{i\in I}\in \mathcal{V}_\eps(O)\right\},
\end{align*}
and define $\M^\ast_H(u;\cdot):\O(\Omega)\to [0,\infty]$ by
\begin{align*}
\M^\ast_H(u;O):=\sup_{\eps> 0}\M^\eps_H(u;O)=\lim_{\eps\to 0}\M^\eps_H(u;O).
\end{align*}
The set function $\M^\ast_H$ is of the Carath\'eodory construction type (see for instance \cite[2.10]{federer69}), which was introduced by \cite{bouchitte-fonseca-mascarenhas98} and \cite{bellieud-bouchitte00}.
\begin{lemma}\label{mastinquelity} Let $O\in\O(\Omega)$. Assume that $H(u;\cdot)$ is countably subadditive for all $u\in W^{1,p}(O;\RR^m)$. Then for every $u\in {W}^{1,p}(O;\RR^m)$ we have
\begin{align}\label{eq01: localsup}
\M_H(u;O)\le\M^\ast_H(u;O).
\end{align}
\end{lemma}
\begin{proof} Fix $\eps>0$. Choose $\{\overline{\Q}_i\}_{i\ge 1}\in \mathcal{V}_\eps(O)$ such that
\begin{align}\label{eq1:localin}
 \sum_{i\ge 1} \M_H(u;\Q_i)\le \frac\eps2+\M_H^\ast(u;O).
\end{align}
For each $i\ge 1$ there exists $\varphi_\eps^i\in {W}^{1,p}_0(\Q_i;\RR^m)$ such that
\begin{align}\label{eq2:localin}
 H(u+\varphi_\eps^i;\Q_i)\le \frac{\eps}{2^{i+1}}+\M_H(u;\Q_i).
\end{align}
Set $\varphi_\eps:=\sum_{i\ge 1}\varphi_\eps^i\mathbb{I}_{\Q_i}\in  {W}^{1,p}_0(O;\RR^m)$. Using the countable subadditivity of $H(u;\cdot)$, \eqref{eq2:localin}, and \eqref{eq1:localin} we have
\begin{align*}
\M_H(u;O)\le H(u+\varphi_\eps;O)\le \sum_{i\ge 1} H(u+\varphi_\eps^i;\Q_i)&\le \frac\eps2+\sum_{i\ge 1} \M_H(u;\Q_i)\\
&\le \eps+\M^\ast_H(u;O),
\end{align*}
we obtain \eqref{eq01: localsup} by letting $\eps\to 0$.
\end{proof}
By \cite[Prop. 2.1., p. 81]{bellieud-bouchitte00}, we have the following result (which is needed for the proof of Lemma~\ref{localdirichlet}).
\begin{lemma}\label{mast:measure}Let $u\in W^{1,p}(\Omega;\RR^m)$. If there exists a finite Radon measure $\mu_u$ on $\Omega$ such that for every cube $\Q\in\O(\Omega)$
\begin{align*}
\M_H(u;\Q)\le \mu_u(\Q),
\end{align*}
then $\M^\ast_H(u;\cdot)$ can be extended to a Radon measure $\lambda_u$ on $\Omega$ satisfying $0\le\lambda_u\le \mu_u$.
\end{lemma}

The proof of the upper bound will be divided into four steps.

\subsection{Step 1: ${\F(u;O)\le\M^\ast_F(u;O)}$ for all ${(u,O)\in W^{1,p}(\Omega;\RR^m)\times\O(\Omega)}$.}\label{subsect1} Fix $u\in W^{1,p}(\Omega;\RR^m)$. Without loss of generality we assume that $\M^\ast_F(u;O)\sinf\infty$. Fix $\eps\in ]0,1[$. Choose $\{\overline{\Q}_i\}_{i\in I}\in\mathcal{V}_\eps(O)$ such that
\begin{align}\label{eq1: step1}
\sum_{i\in I}\M_F(u;\Q_i)\le \M^\eps_F(u;O)+\frac\eps2\le \M^\ast_F(u;O)+\frac\eps2.
\end{align}
Given any $i\in I$ there exists $v_i\in u+{W}^{1,p}_0(\Q_i;\RR^m)$ such that 
\begin{align}\label{eq2: step1}
F(v_i;\Q_i)\le \M^\eps_F(u;\Q_i)+\frac\eps2 \frac{\vert \Q_i\vert}{\vert O\vert}
\end{align}
by definition of $\M_F(u;\Q_i)$.
Define $u_\eps\in u+W^{1,p}_0(O;\RR^m)$ by $u_\eps:=\sum_{i\in I} v_i\mathbb{I}_{\Q_i}+u\mathbb{I}_{\Omega\setminus\mathop{\cup}_{i\in I}\Q_i}$. From \eqref{eq1: step1} and \eqref{eq2: step1} we have that
\begin{align}\label{eq4: step1}
F(u_\eps;O)\le \M^\eps_F(u;O)+\eps.
\end{align}
In the case $p\in]d,\infty[$, from the $p$-coercivity of $f$, \eqref{eq1: step1} and \eqref{eq2: step1}, we deduce
\begin{align}\label{eq3: step1}
\sup_{\eps>0}\int_O \vert\nabla u_\eps\vert^pdx\le \frac{1}{c}\left(\M^\ast_F(u;O)+1\right).
\end{align}
By Poincar\'e inequality there exists $K\ssup 0$ depending only on $p$ and $d$ such that for each $v_i\in u+{W}^{1,p}_0(\Q_i;\RR^m)$
\begin{align*}
\int_{\Q_i}\vert v_i-u\vert^p dx\le K\eps^p\int_{\Q_i}\vert \nabla v_i-\nabla u\vert^pdx,
\end{align*}
since $\diam(\Q_i)\sinf\eps$. By summing on $i\in I$ and using \eqref{eq3: step1} and we obtain
\begin{align*}
\int_O\vert u_\eps-u\vert^p dx&\le 2^{p-1}K\eps^p\left(\int_{O}\vert \nabla u_\eps\vert^p dx+\int_O \vert\nabla u\vert^pdx\right)\\
&\le 2^{p-1}K\eps^p\left(\frac{1}{c}\left(\M^\ast_F(u;O)+1\right)+\int_O \vert\nabla u\vert^pdx\right)
\end{align*}
which shows that $u_\eps\to u$ in $L^p(O;\RR^m)$ as $\eps\to 0$. In the case where $p=\infty$ we have 
\begin{align}\label{eq3: step1prime}
\Vert\nabla u_\eps\Vert_{L^\infty(O;\RR^m)}\le R_0
\end{align} 
since  \eqref{eq4: step1}. With similar reasoning we obtain $u_\eps\to u$ in $L^\infty(O;\RR^m)$ as $\eps\to 0$.

Therefore by \eqref{eq3: step1} (\eqref{eq3: step1prime} if $p=\infty$), there is a subsequence (not relabeled) such that $u_\eps\wto u$ ($u_\eps\stackrel{\ast}{\wto} u$ if $p=\infty$) as $\eps\to 0$, and then by \eqref{eq4: step1} we have
\begin{align*}
\F(u;O)\le \liminf_{\eps\to 0}\F(u_\eps;O)\le \M^\ast_F(u;O).
\end{align*}\hfill$\blacksquare$
\begin{remark} We note that the previous proof establishes
\begin{align*}
\F(u;O)&\le\inf\left\{\liminf_{\eps\to 0}F(u_\eps;O):u+W^{1,p}_0(O;\RR^m)\ni u_\eps\wto u\mbox{ in }W^{1,p}(\Omega;\RR^m)\right\}\\
&\le \M^\ast_F(u;O).
\end{align*}
\end{remark}
%%%%%Step 2
\subsection{Step 2: $\M^\ast_F(u;\cdot)$ is locally equivalent to $\M_F(u;\cdot)$}\label{subsect2}
We are concerned with the proof of the local equivalence of $\M^\ast_F(u;\cdot)$ and $\M_F(u;\cdot)$, this result was established by \cite[Lemma 3.5]{bouchitte-fonseca-mascarenhas98} in the context of relaxation of variational functionals in $BV$, and in a general framework in \cite[Theorem 2.3]{bellieud-bouchitte00}. However, the proof that we propose is inspired by \cite[Proof of Theorem 3.11, p. 380]{acerbi-bouchitte-fonseca03}. Also, note that by Lemma~\ref{eq01: localsup} $\liminf_{\eps\to 0}\frac{\M_F^\ast(u;\Q_\eps(x_0))}{\M_F(u;\Q_\eps(x_0))}\ge 1$.
\begin{lemma}\label{localdirichlet}
If $F(u;O)\sinf\infty$. Then we have 
\[
\displaystyle\lim_{\eps\to 0}\frac{\M^\ast_F(u;\Q_\eps(x_0))}{\eps^d}=\lim_{\eps\to 0}\frac{\M_F(u;\Q_\eps(x_0))}{\eps^d} \mbox{ a.e. in }O.
\]
\end{lemma}
\begin{proof}
Let $u\in W^{1,p}(\Omega;\RR^m)$ be such that $F(u;O)\sinf\infty$. Then for each $U\in\O(O)$
\begin{align*}
\M_F(u;U)\le \int_U \W(x,\nabla u(x))dx\sinf\infty,
\end{align*}
so by using Lemma~\ref{mast:measure} with $\mu_u:=\W(\cdot,\nabla u(\cdot))dx\lfloor_{O}$, $\M^\ast_F(u;\cdot)$ is the trace of a Radon measure $\lambda_u$ on $O$ satisfying $0\le\lambda_u\le \mu_u$. Since $\mu_u$ is absolutely continuous with respect to $dx\lfloor_{O}$ the Lebesgue measure on $O$, the limit $\lim_{\eps\to 0}\frac{\lambda_u(\Q_\eps(x_0))}{\eps^d}$ exists for a.a. $x_0\in O$ as the Radon-Nikodym derivative of $\lambda_u$ with respect to $dx\lfloor_{O}$. Moreover, by Lemma~\ref{mastinquelity}, we have
\begin{align*}
\lim_{\eps\to 0}\frac{\M^\ast_F(u;\Q_\eps(x_0))}{\eps^d}\ge \limsup_{\eps\to 0}\frac{\M_F(u;\Q_\eps(x_0))}{\eps^d}\; \mbox{ a.e. in }O.
\end{align*}
It remains to prove that 
\begin{align}\label{eq-1:localsup}
\lim_{\eps\to 0}\frac{\M^\ast_F(u;\Q_\eps(x_0))}{\eps^d}\le \liminf_{\eps\to 0}\frac{\M_F(u;\Q_\eps(x_0))}{\eps^d}\; \mbox{ a.e. in }O.
\end{align}
Fix any $\theta\ssup 0$. Consider the following sets
\begin{align*}
\mathcal{G}_\theta:=&\Big\{ \Q_\eps(x): x\in O,\;\eps >0\;\mbox{ and }\;\M^\ast_F(u;\Q_\eps(x))\ssup \M_F(u;\Q_\eps(x))+\theta\left\vert\Q_\eps(x)\right\vert\Big\},\\
\mathcal{N}_\theta:=&\Big\{x\in O:\forall \delta\ssup0\;\;\exists\eps \in ]0,\delta[\;\;\Q_\eps(x)\in\mathcal{G}_\theta\Big\}.
\end{align*}
It is sufficient to prove that $\mathcal{N}_\theta$ is a negligible set for the Lebesgue measure on $ O$. Indeed, given $x_0\in O\ssetminus\mathcal{N}_\theta$ there exists $\delta_0\ssup 0$ such that $\M^\ast_F(u;\Q_{\eps}(x_0))\le \M_F(u;\Q_{\eps}(x_0))+\theta\left\vert\Q_{\eps}(x)\right\vert$ for all $\eps\in]0,\delta_0[$. Hence 
\begin{align*}
\displaystyle\lim_{\eps\to 0}\frac{\M^\ast_F(u;\Q_\eps(x_0))}{\left\vert\Q_\eps(x)\right\vert}\le\liminf_{\eps\to 0}\frac{\M_F(u;\Q_{\eps}(x_0))}{\left\vert\Q_{\eps}(x)\right\vert}+\theta,
\end{align*}
then we obtain \eqref{eq-1:localsup} by letting $\theta\to 0$.

Fix $\delta\ssup 0$. Consider the set 
\begin{align*}
\mathcal{F}_\delta:=\Big\{\overline{\Q}_\eps(x): x\in\mathcal{N}_\theta,\;\eps \in]0,\delta[\;\mbox{ and }\Q_\eps(x)\in\mathcal{G}_\theta\Big\}.
\end{align*}
Using the definition of $\mathcal{N}_\theta$ we can see that $\inf_{\Q\in\mathcal{F}_\delta}\diam\left({\Q}\right)=0$. By the Vitali covering theorem there exists a disjointed countable subfamily $\{\overline{\Q}_i\}_{i\ge 1}$ of $\mathcal{F}_\delta$ such that
\begin{align}\label{eq1: localsup}
\big\vert\mathcal{N}_\theta\ssetminus \mathop{\cup}_{i\ge 1}\Q_i\big\vert=0.
\end{align}
We have $\mathcal{N}_\theta\subset \mathop{\cup}_{i\ge 1}\Q_i\cup \mathcal{N}_\theta\ssetminus \mathop{\cup}_{i\ge 1}\Q_i$. To prove that $\mathcal{N}_\theta$ is a negligible set is equivalent to prove that $\vert V_j\vert=0$ for all $j\ge 1$ where 
\[
V_j:=\mathop{\cup}_{i=1}^j\Q_i.
\]
Fix $j\ge 1$. Let $\{\Q^\prime_i\}_{i\ge 1}\in\mathcal{V}_\delta\big( O\setminus \mathop{\cup}_{i=1}^j\overline{\Q}_i\big)$ satisfying
\begin{align}\label{eq2: localsup}
\sum_{i\ge 1}\M_F(u;\Q^\prime_i)\le\M^\ast_F\big(u; O\setminus \mathop{\cup}_{i=1}^j\overline{\Q}_i\big)+\delta.
\end{align}
Recalling that $\M^\ast_F(u;\cdot)$ is the trace on $\O( O)$ of a nonnegative finite Radon measure, we see that 
\begin{align*}
\M^\ast_F(u; O)&\ge \M^\ast_F\big(u; O\setminus \mathop{\cup}_{i=1}^{j}\overline{\Q}_i\big)+\M^\ast_F\big(u;V_j\big)\\
&=\M^\ast_F\big(u; O\setminus \mathop{\cup}_{i=1}^{j}\overline{\Q}_i\big)+\sum_{1\le i\le j}\M^\ast_F(u;{\Q_i}).
\end{align*}
Since each $\Q_i\in\mathcal{G}_\theta$, we have by using \eqref{eq2: localsup}
\begin{align*}
\M^\ast_F(u; O)\ge\sum_{i\ge 1}\M_F(u;\Q^\prime_i)-\delta+\sum_{i=1}^{j}\M_F(u;{\Q_i})+\theta\vert V_j\vert.
\end{align*}
It is easy to see that the countable family $\{\Q^\prime_i:i\ge 1\}\cup\{\Q_i:1\le i\le j\}$ belongs to $\mathcal{V}_\delta( O)$, thus
\begin{align*}
\M^\ast_F(u; O)\ge \M^\delta_F(u; O)+\theta\vert V_j\vert-\delta.
\end{align*}
Letting $\delta\to 0$ we have $\M^\delta_F(u; O)\nearrow \M^\ast_F(u; O)$, and so $\vert V_j\vert=0$ since $\theta\ssup 0$.
\end{proof}
%%%%%%%step3
\subsection{Step 3: Cut-off technique to substitute $u(\cdot)$ with $u(x_0)+\nabla u(x_0)(\cdot-x_0)$ in $\M_F(u;\cdot)$}\label{subsect3}
Now we use cut-off functions to show that for almost all $x_0\in\Omega$ we can replace $u$ in $\M_F(u;\cdot)$ (locally) with the affine tangent map of $u$ at $x_0$ denoted by $u_{x_0}(\cdot):=u(x_0)+\nabla u(x_0)(\cdot-x_0)$. In the following, we consider $u\in W^{1,p}(\Omega;\RR^m)$ satisfying $tu\in\dom F$ for all $t\in ]0,1[$. 

We claim that for every $t\in ]0,1[$
\begin{align}\label{step3-upper-bound}
\lim_{\eps\to 0}\frac{\M_F(tu;\Q_\eps(x_0))}{\eps^d}\le\Z\W(x_0,t\nabla u(x_0))\;\mbox{ a.e. }x_0\in\Omega.
\end{align}
Fix $t\in ]0,1[$ and consider $\lambda,\alpha\in]0,1[$ such that $\lambda=\frac t\alpha$. Fix $x_0\in\Omega$ such that 
\begin{align}
&\lim_{\eps\to 0}\frac{F(\alpha u;\Q_\eps(x_0))}{\eps^d}=\W(x_0,\alpha\nabla u(x_0))\sinf\infty;\label{eq1:cutubound}\\
&\lim_{\eps\to 0}\frac{F(tu;\Q_\eps(x_0))}{\eps^d}=\W(x_0,t\nabla u(x_0))\sinf\infty;\label{eq2:cutubound}\\
&\maxw(x_0)=\lim_{\eps\to 0}\fint_{\Q_\eps(x_0)} \maxw(x)dx\sinf\infty.\label{eq3:cutubound}\\
&\Z \W(x_0,t\nabla u(x_0))=\lim_{\eps\to 0}\inf_{\varphi\in W^{1,p}_0(\Q_\eps(x_0);\RR^m)}\fint_{\Q_\eps(x_0)}\W(y,t\nabla u(x_0)+\nabla \varphi)dy.\label{limZLx0}
\end{align}
To shorten notation the cube $\Q_\eps(x_0)$ is denoted by $\Q_\eps$ .

Let $\{\eps_n\}_{n}\subset \RR_+^\ast$ be a sequence such that $\eps_n\to 0$ as $n\to \infty$ and
\begin{align*}
\lim_{\eps\to 0}\frac{\M_F(tu;\Q_\eps)}{\eps^d}=\lim_{n\to\infty}\frac{\M_F(tu;\Q_{\eps_n})}{\eps_n^d}.
\end{align*}
Fix $r,s\in ]0,1[$ such that $s\sinf r$. Fix $n\ge 1$. Choose $v_n\in  u_{x_0}+{W}^{1,p}_0(\Q_{s\eps_n};\RR^m)$ such that
\begin{align*}
F\left(tv_{n};\Q_{s\eps_n}\right)\le \M_F\left(t u_{x_0};\Q_{s\eps_n}\right)+\left(\eps_n\right)^{d+1},
\end{align*}
Consider a cut-off function $\phi\in W^{1,\infty}_0(\Q_{\eps_n};[0,1])$ such that $\Vert\nabla \phi\Vert_{L^\infty(\Q_{\eps_n})}\le \frac{4}{(r-s)\eps_n}$ and
\begin{align*}
\phi(x)=\left\{
\begin{array}{ll}
  1   & \mbox{ on }\Q_{s\eps_n}\\
  0   & \mbox{ on }\Q_{\eps_n}\ssetminus \Q_{r\eps_n}
\end{array},
\right.
[0<\phi<1]\subsubset \Q_{r\eps_n}\ssetminus \Q_{{s}\eps_n}.
\end{align*}
Define $w_{n}:=\phi v_{n}+(1-\phi)u\in u+{W}^{1,p}_0(\Q_{\eps_n};\RR^m)$, we have
\begin{align}\label{Eq: local limsup}
\M_F\left(tu;\Q_{\eps_n}\right)
&\le F\left(tv_n;\Q_{s\eps_n}\right)+F\left(tw_{n};\Q_{r\eps_n}\ssetminus \Q_{s\eps_n}\right)+F\left(tu;\Q_{\eps_n}\ssetminus \Q_{r\eps_n}\right)\\
&\le\M_F\left(t u_{x_0};\Q_{s\eps_n}\right)+\eps_n^{d+1}+F\left( tw_n;\Q_{r\eps_n}\ssetminus \Q_{s\eps_n}\right)+F\left(tu;\Q_{\eps_n}\ssetminus \Q_{r\eps_n}\right).\notag
\end{align}
The rest of the proof consists to give estimates from above, as $n\to\infty$, of the last two terms of \eqref{Eq: local limsup} divided by $\eps_n^d$.

By \eqref{eq2:cutubound} we have
\begin{align}\label{eq-1: reduc}
\lim_{n\to \infty}\frac{F\left(tu;\Q_{\eps_n}\ssetminus \Q_{r\eps_n}\right) }{\eps_n^d}
=\lim_{n\to \infty}\frac{1}{\eps_n^d}\int_{\Q_{\eps_n}\ssetminus \Q_{r\eps_n}}\nspace\W(x,t\nabla u)dx
=\left(1-r^d\right)\W(x_0,t\nabla u(x_0)).%\notag
\end{align}

By \ref{H3} we have 
\begin{align}\label{eq-2: reduc}
&\frac{F\left(tw_n;\Q_{r\eps_n}\ssetminus \Q_{s\eps_n}\right) }{\eps_n^d}\\
&=  \frac{1}{\eps_n^d}\int_{\Q_{r\eps_n}\ssetminus \Q_{s\eps_n}} \nspace\W\left(x,\lambda(\psi\alpha\nabla u(x_0)+(1-\psi)\alpha\nabla u)+(1-\lambda)\Phi_{n,t}\right)dx\notag\\
&\le C_1  \left(\frac{\left\vert{\Q_{r\eps_n}\ssetminus \Q_{s\eps_n}}\right\vert}{\eps_n^d}+\frac{1}{\eps_n^d}\int_{\Q_{r\eps_n}\ssetminus \Q_{s\eps_n}}\nspace\W\left(x,{\alpha}\nabla u(x_0)\right)dx\right.\notag\\
&\hspace{2cm}+\left. \frac{1}{\eps_n^d}\int_{\Q_{r\eps_n}\ssetminus \Q_{s\eps_n}}\nspace\W\left(x,{\alpha}\nabla u\right)dx+\frac{1}{\eps_n^d}\int_{\Q_{r\eps_n}\ssetminus \Q_{s\eps_n}}\nspace\W\left(x,\Phi_{n,t}\right)dx\right),\notag
\end{align}
where $\Phi_{n,t}:=\frac{t}{1-{\lambda}}\nabla \phi\otimes( u_{x_0}-u)$ and $C_1=C^2+C$. By \eqref{eq1:cutubound}, it holds
\begin{align}\label{eq1: reduc}
\displaystyle \lim_{n\to \infty} \frac{1}{\eps_n^d}\int_{\Q_{r\eps_n}\ssetminus \Q_{s\eps_n}}\nspace\W\left(x,{\alpha}\nabla u\right)dx=(r^d-s^d)f(x_0,\alpha \nabla u(x_0)).
\end{align}
Using \ref{H4} and \eqref{eq1:cutubound} we deduce
\begin{align}\label{eq2: reduc}
\lim_{n\to \infty}\frac{1}{\eps_n^d}\int_{\Q_{r\eps_n}\ssetminus \Q_{s\eps_n}}\nspace\W\left(x,{\alpha}\nabla u(x_0)\right)dx=(r^d-s^d)\a(x_0,\alpha \nabla u(x_0)).
\end{align}
Choose $N_0\ge 1$ such that $\frac{1}{\eps_n}\Vert  u_{x_0}-u\Vert_{L^\infty(\Q_{\eps_n})}\le \frac{(1-\lambda)(r-s)\rho_0}{4t}$ for all $n\ge N_0$. It follows that $\Vert \Phi_{n,t}\Vert_{L^\infty(\Q_{\eps_n};\MM^{md})}\le \rho_0$ for all $n\ge N_0$, and we have by Lemma~\ref{lemma finite integrand}
\begin{align*}
\limsup_{n\to\infty}\frac{1}{\eps_n^d}\int_{\Q_{r\eps_n}\ssetminus \Q_{s\eps_n}}\nspace\W\left(x,\Phi_{n,t}\right)dx\le \limsup_{n\to\infty}\frac{1}{\eps_n^d}\int_{\Q_{r\eps_n}\ssetminus \Q_{s\eps_n}}\nspace \maxw(x) dx.
\end{align*}
Moreover, by \eqref{eq3:cutubound} we have
\begin{align}\label{eq3: reduc}
\limsup_{n\to \infty}\frac{1}{\eps_n^d}\int_{\Q_{r\eps_n}\ssetminus \Q_{s\eps_n}}\nspace \maxw(x) dx=(r^d-s^d)\maxw(x_0).
\end{align}
Passing to the limit $n\to\infty$ by taking account of \eqref{eq-2: reduc}, and the estimates \eqref{eq1: reduc}, \eqref{eq2: reduc} and \eqref{eq3: reduc}, we have
\begin{align*}
&\lim_{n\to\infty}\frac{\M_F\left(tu;\Q_{\eps_n}\right)}{\eps_n^d}\\
&\le s^d\liminf_{n\to\infty}\frac{\M_F\left(t u_{x_0};\Q_{s\eps_n}\right)}{s^d\eps_n^d}+\left(1-r^d\right)\W(x_0,t\nabla u(x_0))\\
&+2C_1(r^d-s^d)\left(1+\W(x_0,\alpha\nabla u(x_0))+\a(x_0,\alpha\nabla u(x_0))+\maxw(x_0)\right).
\end{align*}
Letting $r\to 1$ and $s\to 1$, we find 
\begin{align}\label{end-eq-upper-bound}
\lim_{\eps\to 0}\frac{\M_F\left(tu;\Q_{\eps}\right)}{\eps^d}=\lim_{n\to\infty}\frac{\M_F\left(tu;\Q_{\eps_n}\right)}{\eps_n^d}&\le \liminf_{s\to 1}\liminf_{n\to\infty}\frac{\M_F\left(t u_{x_0};\Q_{s\eps_n}\right)}{(s\eps_n)^d}\\
&\le \limsup_{\eps\to 0}\frac{\M_F\left(t u_{x_0};\Q_{\eps}\right)}{\eps^d}\notag\\&=\Z\W(x_0,t\nabla u(x_0))\notag
\end{align}
where \eqref{limZLx0} is used.

\subsection{Step 4: End of the proof of Proposition~\ref{main proposition} (i)} Using in turn the results of Subsect.~\ref{subsect1},~\ref{subsect2} and~\ref{subsect3} we obtain for every $u\in W^{1,p}(\Omega;\RR^m)$, every $O\in\O(\Omega)$ and every $t\in]0,1[$
\begin{align*}
\F(tu;O)\le \M_F^\ast(tu;O)=\int_O \lim_{\eps\to 0}\frac{\M_F^\ast\left(tu;\Q_{\eps}(x)\right)}{\eps^d}dx&=\int_O\lim_{\eps\to 0}\frac{\M_F\left(tu;\Q_{\eps}(x)\right)}{\eps^d}dx\\
&\le\int_O \Z\W(x,t\nabla u(x))dx.
\end{align*}\hfill$\blacksquare$
%%%%%%%%%%%%%
\section{Proof of Proposition~\ref{main proposition} {\rm (ii)} and {\rm (iii)}}\label{Proof of Proposition (iii)}
The proof will be divided into two steps. In the first step we will use a localization technique also known as blow-up method introduced by \cite{fonseca-muller92} which consists to reduce the proof of the (global) lower bound to a local lower bound by using measure arguments. The second step consists to prove the local lower bound by using cut-off functions.

In this section we denote by $L$ the integrands $\Z\W$ or$\W$. 
\subsection{Step 1: Localization technique} Let $O\in\O(\Omega)$. Let $u,\{u_n\}_n\subset W^{1,p}(\Omega;\RR^m)$ be such that $u_n\wto u$ in $W^{1,p}(\Omega;\RR^m)$ and 
\begin{align*}
\infty\ssup\L(u;O):=&\inf\left\{\liminf_{n\to \infty}\int_O L(x,\nabla u_n(x))dx: W^{1,p}(\Omega;\RR^m)\ni u_n\wto u\right\}\\
=&\lim_{n\to\infty}\int_O L(x,\nabla u_n(x))dx.
\end{align*}
Up to a subsequence, since $p\ssup d$, we may assume that 
\begin{align}\label{CONVU}
u_n\to u\mbox{ in }L^\infty(\Omega;\RR^m).
\end{align}
Passing to a subsequence if necessary, we may find a nonnegative Radon measure $\mu$ such that
\[
L(\cdot,\nabla u_n(\cdot))dx_{\lfloor_{O}}\swto \mu\;\mbox{ as }n\to \infty\mbox{ weakly $\ast$ in the sense of measures.}
\]
It is enough to prove that for all $t\in ]0,1[$
\begin{align}\label{CLAIM}
\frac{d\mu}{dx}(\cdot)+\Delta_L^a(t)\left(a(\cdot)+\frac{d\mu}{dx}(\cdot)\right)\ge \Z L(\cdot,t\nabla u(\cdot)) \mbox{ a.e. in }O.
\end{align}
Indeed, by Alexandrov theorem, we will have
\begin{align*}
\L(u;O)=\lim_{n\to \infty}L(u_n;O)\ge \liminf_{n\to \infty}\int_O L(x,\nabla u_n)dx=\mu(O)\ge\int_O \frac{d\mu}{dx}(x)dx,
\end{align*}
so by integrating over $O$ in \eqref{CLAIM}, we find 
\begin{align*}
\L(u;O)+\Delta_L^a(t)\left(\vert a\vert_{L^1(O)}+\L(u;O)\right)\ge\int_O \Z L(x,t\nabla u(x))dx.
\end{align*}
As $L$ is ru-usc, we obtain the result by passing to the limit $t\to 1$ and by using Fatou lemma.

Since $\int_\Omega \W(x,\nabla u(x))dx\sinf\infty$, we fix $x_0\in O$ such that \ref{H4} holds and
\begin{align}
&\W(x_0,\nabla u(x_0))\sinf\infty;\label{Blow-up: eq0}\\
&L(x_0,\nabla u(x_0))\le \W(x_0,\nabla u(x_0))\sinf\infty;\label{Blow-up:eq01}\\
&\frac{d\mu}{dx}(x_0)=\lim_{\eps\to 0}\frac{\mu(Q_\eps(x_0))}{\eps^d}\sinf\infty;\label{Blow-up:eq1}\\
&\lim_{\eps\to 0} \frac{1}{\eps}\left\Vert u-u(x_0)-\nabla u(x_0)(\cdot-x_0)\right\Vert_{L^\infty(Q_\eps(x_0);\RR^m)}=0;\label{Blow-up:eq2}\\
&a(x_0)=\lim_{\eps\to 0}\fint_{Q_\eps(x_0)} a(x)dx\sinf\infty; \label{Blow-up: eq-1}\\
&\maxw(x_0)=\lim_{\eps\to 0}\fint_{Q_\eps(x_0)} \maxw(x)dx\sinf\infty; \label{Blow-up: eq-2}
\end{align}
where $Q_\eps(x_0):=x_0+\eps Y$. Note that \eqref{Blow-up: eq-2} is a consequence of Lemma~\ref{lemma finite integrand}. 

\medskip

Choose $\eps_k\to 0$ such that  
$
\mu(\partial Q_{\eps_k}(x_0))=0.
$
Then 
\begin{align}\label{extract: eq1}
\lim_{k\to \infty}\frac{\mu(Q_{\eps_k}(x_0))}{\eps^d_k}&=\lim_{k\to \infty}\lim_{n\to \infty}\fint_{Q_{\eps_k}(x_0)}L(x,\nabla u_n)dx\\
&=\lim_{k\to \infty}\lim_{n\to \infty}\int_Y L(x_0+\eps_k y,\nabla v_{n,k})dy,\notag
\end{align}
where $v_{n,k}(y):=\frac{u_n(x_0+\eps_ky)-u(x_0)}{\eps_k}$. By \eqref{CONVU} we have
\begin{align}\label{almost diff: eq1}
\lim_{k\to\infty}\lim_{n\to \infty}\left\Vert v_{n,k}-l_{\nabla u(x_0)}\right\Vert_{L^\infty(Y;\RR^m)}=0\quad\mbox{ where }\quad l_{\nabla u(x_0)}(y):=\nabla u(x_0)y.
\end{align}
Fix $s,r\in ]0,1[$ such that $s\sinf r$. Then \eqref{Blow-up:eq1} implies (see Subsect.~\ref{measure-rs} for the proof)
\begin{align}\label{extract: eq4}
\lim_{k\to\infty}\limsup_{n\to\infty}\frac{\mu_n\left(Q_{r\eps_k}(x_0)\ssetminus \overline{Q}_{s\eps_k}(x_0)\right)}{\eps_k^d}=(r^d-s^d)\frac{d\mu}{dx}(x_0).
\end{align}
By a simultaneous diagonalization of \eqref{extract: eq1}, \eqref{almost diff: eq1} and \eqref{extract: eq4}, we may extract a subsequence $v_n:=v_{n,k_n}$ satisfying
\begin{align}
&v_n\to l_{\nabla u(x_0)}\mbox{ in }L^\infty(Y;\RR^m),\quad v_n\wto l_{\nabla u(x_0)} \mbox{ in }W^{1,p}(Y;\RR^m),\label{HYPCUT1}\\
&\frac{d\mu}{dx}(x_0)=\lim_{n\to \infty}\int_Y L(x_0+\eps_n y,\nabla v_n)dy,\label{HYPCUT2}\\
&(r^d-s^d)\frac{d\mu}{dx}(x_0)= \lim_{n\to\infty}\frac{\mu_n\left(Q_{r\eps_n}(x_0)\ssetminus \overline{Q}_{s\eps_n}(x_0)\right)}{\eps_n^d},\label{diagonal2}
\end{align}
where $\eps_{k_n}:=\eps_n$.
%%%%%{Cut-off method} 
\subsection{Step 2: Cut-off technique to substitute $v_n$ with $w_n\in l_{\nabla u(x_0)}+W^{1,p}(Y;\RR^m)$} 
For simplicity of notation we set $\theta_{x_0,n}(y):=x_0+\eps_ny$ for all $y\in Y$. 
In this section we use cut-off functions to show that there exists $\{w_n\}_n\subset l_{\nabla u(x_0)}+W^{1,p}_0(Y;\RR^m)$ such that for every $t\in ]0,1[$
\begin{align}\label{CUTOFF}
\limsup_{n\to \infty}\int_Y L(\theta_{x_0,n}, t\nabla w_n)dy\le \frac{d\mu}{dx}(x_0)+\Delta_L^a(t)\left(a(x_0)+\frac{d\mu}{dx}(x_0)\right).
\end{align}
If \eqref{CUTOFF} holds then 
\begin{align*}
\Z L(x_0,t\nabla u(x_0))&\le\liminf_{\eps\to 0}\inf\left\{\int_Y L(x_0+\eps y,t\nabla w)dy:w\in l_{\nabla u(x_0)}+W^{1,p}_0(Y;\RR^m)\right\}\\
&\le\limsup_{n\to \infty}\int_Y L(\theta_{x_0,n},t\nabla w_n)dy\\
&=\frac{d\mu}{dx}(x_0)+\Delta_L^a(t)\left(a(x_0)+\frac{d\mu}{dx}(x_0)\right),
\end{align*}
and the claim~\ref{CLAIM} follows.

Now, let us prove \eqref{CUTOFF}. Fix any $t\in]0,1[$. Let $\phi\in W^{1,\infty}_0(Y;[0,1])$ be a cut-off function between ${s} \overline{Y}$ and $\overline{Y}\ssetminus r{Y}$ such that $\|\nabla\phi\|_{L^\infty(Y)}\leq{4\over r-s}$. Setting
\[
w_n:=\phi v_n+(1-\phi)l_{\nabla u(x_0)}.
\]
We have $w_n\in l_{\nabla u(x_0)}+W^{1,p}_0(Y;\RR^m)$ and
\begin{align*}
\nabla w_n:=\left\{
\begin{array}{ll}
\nabla v_n&\mbox{ on }{s}Y\\ \\
\phi\nabla v_n+(1-\phi)\nabla u(x_0)+\Phi_{n,s,r}&\mbox{ on }U_{s,r} \\\\
{\nabla u(x_0)}&\mbox{ on }Y\ssetminus r \overline{Y},
\end{array}
\right.
\end{align*}
where $\Phi_{n,s,r}:=\nabla\phi\otimes\left(v_n-l_{\nabla u(x_0)}\right)$ and $U_{s,r}:=r {Y}\ssetminus {s} \overline{Y}$.

For every $n\ge 1$, it holds
\begin{align}\label{blow-up:estimate0}
&\int_Y L(\theta_{x_0,n},t\nabla w_n)dy\\
=&\int_{{s}Y} L(\theta_{x_0,n},t\nabla v_n)dy+\int_{U_{s,r}} L(\theta_{x_0,n},t\nabla w_n)dy+\int_{Y\ssetminus {r}\overline{Y}} L(\theta_{x_0,n},t\nabla u(x_0))dy\notag\\
\le & \int_{Y} L(\theta_{x_0,n},t\nabla v_n)dy+\int_{U_{s,r}} L(\theta_{x_0,n},t\nabla w_n)dy+\int_{Y\ssetminus {r}\overline{Y}} L(\theta_{x_0,n},t\nabla u(x_0))dy.\notag
\end{align}
The rest of the proof consists to give estimates from above, as $n\to\infty$, of the last three terms of \eqref{blow-up:estimate0}.\\

\noindent{\bf Bound for $\limsup_{n\to\infty}\int_{Y} L(\theta_{x_0,n},t\nabla v_n)dy$.} Since $L$ is ru-usc, using \ref{HYPCUT2} and \eqref{Blow-up: eq-1}, we have for every $n\ge 1$ 
\begin{align}\label{blow-up:estimate1}
&\limsup_{n\to\infty}\int_{Y} L(\theta_{x_0,n},t\nabla v_n)dy\\
\le&\limsup_{n\to\infty}\left( \Delta_L^a(t)\int_Ya(\theta_{x_0,n})+L(\theta_{x_0,n},\nabla v_n)dy+\int_{Y} L(\theta_{x_0,n},\nabla v_n)dy\right)\notag\\
\le& \Delta_L^a(t)\left(a(x_0)+\frac{d\mu}{dx}(x_0)\right)+\frac{d\mu}{dx}(x_0).\notag
\end{align}

\noindent{\bf Bound for $\limsup_{n\to\infty}\int_{Y\setminus {r}\overline{Y}} L(\theta_{x_0,n},t\nabla u(x_0))dy$.} Similarly to the previous estimate we have
\begin{align*}
&\limsup_{n\to\infty}\int_{Y\ssetminus {r}\overline{Y}} L(\theta_{x_0,n},t\nabla u(x_0))dy\\
&\le\limsup_{n\to\infty}\left( \Delta_L^a(t)\int_{Y\ssetminus {r}\overline{Y}}a(\theta_{x_0,n})+L(\theta_{x_0,n},\nabla u(x_0))dy+\int_{Y\ssetminus {r}\overline{Y}} L(\theta_{x_0,n},\nabla u(x_0))dy\right)\notag\\
&\le \Delta_L^a(t)\left((1-r^d)a(x_0)+A_r(x_0)\right)+A_r(x_0)\notag
\end{align*}
where $A_r(x_0):=\limsup_{n\to\infty}\int_{Y\ssetminus {r}\overline{Y}} L(\theta_{x_0,n},\nabla u(x_0))dy$. Now, taking account of \eqref{Blow-up: eq0} and \eqref{Blow-up:eq01}, we have, by \ref{H4}, an upper bound for $A_r(x_0)$
\begin{align}\label{blow-up:estimate2bis}
A_r(x_0)\le (1-r^d)\a(x_0,\nabla u(x_0)).
\end{align}
We deduce 
\begin{align}\label{blow-up:estimate2}
\limsup_{n\to\infty}&\int_{Y\setminus {r}\overline{Y}} L(\theta_{x_0,n},t\nabla u(x_0))dy\\
&\le (1-r^d)\left(\Delta_L^a(t)(a(x_0)+L(x_0,\nabla u(x_0)))+\a(x_0,\nabla u(x_0))\right).\notag
\end{align}

\noindent{\bf Bound for $\limsup_{n\to \infty}\int_{U_{s,r}} L(\theta_{x_0,n},t\nabla w_n)dy$.}
%It remains to estimate the middle term of \eqref{blow-up:estimate0}. 
Since $\W$ satisfies \ref{H3}, we have that $L=\Z\W$ also satisfies \ref{H3} by Lemma~\ref{ZL satisfies H3}. Therefore for every $n\ge 1$
\begin{align*}
&\int_{U_{s,r}} L(\theta_{x_0,n},t\nabla w_n)dy\\
&\le C_1\left((r^d-s^d)+\int_{U_{s,r}} L(\theta_{x_0,n},\nabla v_n)dy+\int_{U_{s,r}} L(\theta_{x_0,n},\nabla u(x_0))dy\right.\\
&\left.+\int_{U_{s,r}} L\left(\theta_{x_0,n},\frac{t}{1-t}\Phi_{n,s,r}\right)dy\right)
\end{align*}
where $C_1=C(1+C)$. Since \ref{HYPCUT1}, there exists $N_{0}\ge 1$ such that for every $n\ge N_{0}$
\begin{align*}
\left\Vert{t\over 1-{t}}\Phi_{n,s,r}\right\Vert_{L^\infty(Y;\MM^{m\times d})}\le \rho_0
\end{align*}
where $\rho_0\ssup 0$ is given by Lemma~\ref{lemma finite integrand}. Taking account of \eqref{Blow-up: eq-2}, we have
\begin{align}\label{eq2CUT}
\limsup_{n\to \infty}\int_{U_{s,r}}L\left(\theta_{x_0,n},{t\over 1-{t}}\Phi_{n,s,r}\right)dy&\le \limsup_{n\to \infty}\int_{U_{s,r}}\sup_{\zeta\in\overline{\Q}_{\rho_0}(0)}L(\theta_{x_0,n},\zeta)dy\\
&\le (r^d-s^d)\maxw(x_0).\notag
\end{align}
Using similar reasoning as in estimate \eqref{blow-up:estimate2bis}, we find
\begin{align}\label{eq3CUT}
 \limsup_{n\to \infty}\int_{U_{s,r}} L(\theta_{x_0,n},\nabla u(x_0))dy\le (r^d-s^d)\a(x_0,\nabla u(x_0)).
\end{align} 
Since \eqref{diagonal2}, we have
\begin{align}\label{eq4CUT}
 \limsup_{n\to \infty}\int_{U_{s,r}} L(\theta_{x_0,n},\nabla v_n)dy=(r^d-s^d)\frac{d\mu}{dx}(x_0).
\end{align} 
Collecting \eqref{eq2CUT}, \eqref{eq3CUT} and \eqref{eq4CUT}, we obtain
\begin{align}\label{blow-up:estimate3}
&\limsup_{n\to \infty}\int_{U_{s,r}} L(\theta_{x_0,n},t\nabla w_n)dy\\
&\le C_1(s^d-r^d)\left(1+\frac{d\mu}{dx}(x_0)+\a(x_0,\nabla u(x_0))+\maxw(x_0)\right)\notag.
\end{align}

\noindent{\bf End of the proof of \eqref{CUTOFF}.} Collecting \eqref{blow-up:estimate1}, \eqref{blow-up:estimate2} and \eqref{blow-up:estimate3}, we have
\begin{align*}
&\limsup_{n\to \infty}\int_Y L(\theta_{x_0,n},t\nabla w_n)dy\\
&\le\Delta_L^a(t)\left(a(x_0)+\frac{d\mu}{dx}(x_0)\right)+\frac{d\mu}{dx}(x_0)\\
&+(1-r^d)\big(\Delta_L^a(t)(a(x_0)+L(x_0,\nabla u(x_0)))+\a(x_0,\nabla u(x_0))\big)\\
&+C_1(r^d-s^d)\left(1+\frac{d\mu}{dx}(x_0)+\a(x_0,\nabla u(x_0))+\maxw(x_0)\right).
\end{align*}
we obtain \eqref{CUTOFF} by letting $r\to 1$ and $s\to 1$. $\hfill\blacksquare$

\subsubsection{Proof of \eqref{extract: eq4}}\label{measure-rs} By \eqref{Blow-up:eq1} we have
\begin{align}\label{Blow-up paramater: eq}
\frac{d\mu}{dx}(x_0)=\lim_{\eps\to 0}\frac{\mu(Q_{s\eps}(x_0))}{(s\eps)^d}=\lim_{\eps\to 0}\frac{\mu(Q_{r\eps}(x_0))}{(r\eps)^d}=\lim_{\eps\to 0}\frac{\mu(\overline{Q}_{s\eps}(x_0))}{(s\eps)^d}=\lim_{\eps\to 0}\frac{\mu(\overline{Q}_{r\eps}(x_0))}{(r\eps)^d}.
\end{align}
on one hand we have
\begin{align}
\label{extract: eq2}\liminf_{k\to \infty}\liminf_{n\to\infty}\frac{\mu_n\left(Q_{r\eps_k}(x_0)\ssetminus \overline{Q}_{s\eps_k}(x_0)\right)}{\eps_k^d} &\ge \liminf_{k\to \infty}\frac{\mu\left(Q_{r\eps_k}(x_0)\ssetminus \overline{Q}_{s\eps_k}(x_0)\right)}{\eps_k^d}\\
&= (r^d-s^d)\frac{d\mu}{dx}(x_0)\notag
\end{align}
by using \eqref{Blow-up paramater: eq} and Alexandrov theorem. Similarly, on the other hand we have
\begin{align}
\label{extract: eq3}&\limsup_{k\to\infty}\limsup_{n\to\infty}\frac{\mu_n\left(Q_{r\eps_k}(x_0)\ssetminus \overline{Q}_{s\eps_k}(x_0)\right)}{\eps_k^d}\\
&\le \limsup_{k\to \infty}\limsup_{n\to\infty}\left(r^d \frac{\mu_n(Q_{r\eps_k}(x_0))}{(r\eps_k)^d}- s^d\frac{\mu_n(\overline{Q}_{s\eps_k}(x_0))}{(s\eps_k)^d}\right)\notag\\
&\le \limsup_{k\to \infty}\limsup_{n\to\infty}\left(r^d \frac{\mu_n(\overline{Q}_{r\eps_k}(x_0))}{(r\eps_k)^d}- s^d\frac{\mu_n({Q}_{s\eps_k}(x_0))}{(s\eps_k)^d}\right)\notag\\
&\le \limsup_{k\to \infty}\left(r^d \frac{\mu(\overline{Q}_{r\eps_k}(x_0))}{(r\eps_k)^d}- s^d\frac{\mu({Q}_{s\eps_k}(x_0))}{(s\eps_k)^d}\right)= (r^d-s^d)\frac{d\mu}{dx}(x_0)\notag
\end{align}
Combining \eqref{extract: eq2} and \eqref{extract: eq3}, we obtain \eqref{extract: eq4}. $\hfill\blacksquare$

%%%%%%%%%%%%%%%%%%%Derivation of a set function with respect to the Lebesgue measure%%%%%%%
\section{Proof of Proposition~\ref{exist-ZL}}\label{Proof of Theorem exist-ZL}

We denote by $\Cub$ the family of all open cubes of $\RR^d$. We denote by $\Cub_\delta$ the family of all open cubes $\Q$ of $\RR^d$ such that $\diam(\Q)\sinf \delta$, where $\delta\ssup 0$. For each $E\subset \RR^d$, we associate the set $\FF_\delta(E)$ of all countable families $\{\Q_i\}_{i\in I}\subset\Cub_\delta$ satisfying $\vert E\setminus \mathop{\cup}_{i\in I}\Q_i\vert=0$, $\Q_i\cap E\not=\emptyset$ for all $i\in I$, and $\overline{\Q}_i\cap\overline{\Q}_j=\emptyset$ for all $i\not=j$. If $E\not=\emptyset$ then $\FF_\delta(E)\not=\emptyset$, indeed, by the Vitali covering theorem, it is always possible by starting from a family of closed cubes of $\RR^d$ with center in $E$ to find a countable subfamily of open cubes in $\FF_\delta(E)$ because the Lebesgue measure of the boundary of a cube is null, i.e., $\vert \overline{\Q}\vert=\vert\Q\vert$ for all $\Q\in\Cub$.

Let $\M$ be a nonnegative set function defined for all cubes of $\RR^d$ such that $\M(\emptyset)=0$. Let $\M^\sharp:\mathcal{P}(E)\to [0,\infty]$ be defined by 
\begin{align*}
\M^\sharp(E)&:=\left\{
\begin{array}{cl}
\displaystyle\sup_{\delta>0}\M^\delta(E)&\mbox{ if }E\not=\emptyset\\
0&\mbox{ otherwise,}
\end{array}
\right.
\\
\mbox{ with }\;\;\M^\delta(E)&:=\inf\left\{\sum_{i\in I}\M(\Q_i):\{\Q_i\}_{i\in I}\in\FF_\delta(E)\right\}.
\end{align*}
We denote by $\omega\in [0,\infty]$ the number
\begin{align*}
\omega:=\limsup_{\delta\to 0}\sup_{\substack{\Q\subset\Omega\\ \diam(\Q)<\delta}}\frac{\M(\Q)}{\vert \Q\vert}
\end{align*}
where $\Q$ denotes any arbitrary open cube of $\RR^d$.

\medskip

The following result is an abstract version of Proposition~\ref{exist-ZL}.
\begin{proposition}\label{existlimit} If $\omega\sinf\infty$ and 
\begin{equation}\label{hyp-import}
\limsup_{\delta\to 0}\frac{\M(\Q_\delta(x))}{\delta^d}\le \limsup_{\delta\to 0}\frac{\M^\sharp(\Q_\delta(x))}{\delta^d}\mbox{ a.e. in }\Omega
\end{equation}
then
\begin{align*}
\limsup_{\delta \to 0}\frac{\M(\Q_\delta(x))}{\delta^d}=\liminf_{\delta \to 0}\frac{\M(\Q_\delta(x))}{\delta^d}\mbox{ a.e. in }\Omega
\end{align*}
where $\Q_\delta(x)=x+\delta Y$ for any $x\in\Omega$ and $\delta\ssup 0$.
\end{proposition}
The set function $\M^\sharp$ is of Carath\'eodory type construction (see \cite[Sect. 2.10, p. 169]{federer69}). Although we do not know whether it is an outer measure we have the following result. 
\begin{lemma}\label{outermeasure} The set function $\M^\sharp$ satisfies:
\begin{enumerate}
\item[(i)] if $E_1,E_2\subset\RR^d$ are two sets such that $\dist(E_1,E_2)\ssup 0$ then $\M^\sharp(E_1\cup E_2)=\M^\sharp(E_1)+\M^\sharp(E_2)$;
\item[(ii)] if $E,V\subset\RR^d$ are such that $V$ is a nonempty open set and $E\subset V$ then $\M^\sharp(E)\le \M^\sharp(V)$;
\item[(iii)] if $\omega\sinf\infty$ then $\M^\sharp(E)\le \omega\vert E\vert$ for all closed set $E\subset\Omega$.
\end{enumerate} 
\end{lemma}
\begin{proof}(i) We show that for every $E_1,E_2\subset \RR^d$ satisfying $\dist(E_1,E_2)\ssup \delta_0$ for some $\delta_0\ssup 0$ we have 
\begin{align}\label{outermeasure: eq1}
\M^\sharp(E_1\cup E_2)\ge \M^\sharp(E_1)+\M^\sharp(E_2).
\end{align}
Fix $\delta\in ]0,\delta_0[$. Choose $\{\Q_i\}_{i\in I}\in\Cub_\delta$ satisfying $\vert (E_1\cup E_2) \setminus \mathop{\cup}_{i\in I}\Q_i\vert=0$, $\Q_i\cap (E_1\cup E_2)\not=\emptyset$ for all $i\in I$, and
\begin{align}\label{eqoutermeasure2}
\M^\sharp(E_1\cup E_2)+{\delta}\ge \sum_{i\in I}\M(\Q_i).
\end{align}
Let $I_j=\{i\in I: \Q_i\cap E_j\not=\emptyset\}$ for $j\in\{1,2\}$. Since $\dist(E_1,E_2)\ssup 2\delta$, if $i\in I_1$ (resp. $i\in I_2$) then $\Q_i\cap E_2=\emptyset$ (resp. $\Q_i\cap E_1=\emptyset$). Thus 
\begin{align*}
0=\vert (E_1\cup E_2) \setminus \mathop{\cup}_{i\in I}\Q_i\vert=\vert E_1\setminus\mathop{\cup}_{i\in I_1}\Q_i\vert+\vert E_2\setminus\mathop{\cup}_{i\in I_2}\Q_i\vert,
\end{align*}
hence $\vert E_j\setminus\mathop{\cup}_{i\in I_j}\Q_i\vert=0$ and $\Q_i\cap E_j\not=\emptyset$ for all $i\in I_j$. From \eqref{eqoutermeasure2} we have
\begin{align*}
\M^\sharp(E_1\cup E_2)+{\delta}\ge \sum_{i\in I_1}\M(\Q_i)+\sum_{i\in I_2}\M(\Q_i)\ge \M^\sharp(E_1)+\M^\sharp(E_2),
\end{align*}
and \eqref{outermeasure: eq1} holds by letting $\delta\to 0$. 

Now, we show that 
\begin{align}\label{outermeasure: eq2}
\M^\sharp(E_1\cup E_2)\le \M^\sharp(E_1)+\M^\sharp(E_2).
\end{align}
For each $j\in\{1,2\}$, choose $\{\Q_i^j\}_{i\in I_j}\in\Cub_\delta$ satisfying $\vert E_j \setminus \mathop{\cup}_{i\in I_j}\Q_i^j\vert=0$, $\Q_i^j\cap E_j\not=\emptyset$ for all $i\in I_j$, and
\begin{align}\label{eqoutermeasure3}
\M^\sharp(E_j)+{\delta}\ge \sum_{i\in I_j}\M(\Q_i^j).
\end{align}
Since $\dist(E_1,E_2)\ssup \delta_0$ the countable family of cubes $\{\Q_i^j:{i\in I_j} \mbox{ and }j\in\{1,2\}\}$ is pairwise disjointed, moreover we have 
\begin{align*}
\vert (E_1\cup E_2)\ssetminus \mathop{\cup}_{i\in I_1}\Q_i\cup \mathop{\cup}_{i\in I_2}\Q_i\vert\le \vert E_1\ssetminus \mathop{\cup}_{i\in I_1}\Q_i\vert+\vert E_2\ssetminus \mathop{\cup}_{i\in I_2}\Q_i\vert=0.
\end{align*}
Summing over $j\in \{1,2\}$ in \eqref{eqoutermeasure3} we obtain
\begin{align*}
\M^\sharp(E_1)+\M^\sharp(E_2)+2\delta\ge \sum_{j\in\{1,2\}}\sum_{i\in I_j}\M(\Q_i^j)\ge \M^\sharp(E_1\cup E_2),
\end{align*}
and \eqref{outermeasure: eq2} holds by letting $\delta\to 0$.

\medskip

(ii) Let $E,V$ be two sets of $\RR^d$ such that $V$ is a nonempty open set and $E\subset V$. For each $\delta\ssup 0$ choose $\{\Q_i\}_{i\in I}\in\Cub_\delta$ satisfying $\vert V \ssetminus \mathop{\cup}_{i\in I}\Q_i\vert=0$, $\Q_i\cap V\not=\emptyset$ for all $i\in I$, and
\begin{align}\label{eqoutermeasure4}
\M^\sharp(V)+{\delta}\ge \sum_{i\in I}\M(\Q_i).
\end{align}
Consider the open set $V_\delta:=\mathop{\cup}_{i\in I}\Q_i$, then $\vert V \ssetminus \overline{V_\delta}\vert=0$, but $V \ssetminus \overline{V_\delta}$ is open so $V \ssetminus \overline{V_\delta}=\emptyset $. It means that $V\subset\overline{V_\delta}$ so $V\subset{V_\delta}$. We deduce that $I_E\not=\emptyset$ where $I_E:=\{i\in I:\Q_i\cap E\not=\emptyset\}$. We have $\vert E\ssetminus \mathop{\cup}_{i\in I_E}\Q_i\vert=\vert E\ssetminus \mathop{\cup}_{i\in I}\Q_i\vert\le \vert V \ssetminus \mathop{\cup}_{i\in I}\Q_i\vert =0$, thus $\{\Q_i\}_{i\in I_E}\in \FF_\delta(E)$ and so from \eqref{eqoutermeasure4} 
\begin{align*}
\M^\sharp(V)+{\delta}\ge \M^\sharp(E),
\end{align*}
and (ii) holds by letting $\delta\to 0$.

\medskip

(iii) Fix $\delta\ssup 0$ and $E\subset \Omega$. Set $E_\delta=\{x\in \RR^d:\dist(x,E)\sinf\delta\}$, then for any countable family $\{\Q_i\}_{i\in I}\in\FF_\delta(E)$ we have $\mathop{\cup}_{i\in I}\Q_i\subset E_\delta$ since $\Q_i\cap E\not=\emptyset$ for all $i\in I$ and $\diam(\Q)\sinf\delta$. Therefore
\begin{align*}
\M^\delta(E)\le \sum_{i\in I}\M(\Q_i)\le \sum_{i\in I}\frac{\M(\Q_i)}{\vert\Q_i\vert}\vert\Q_i\vert&\le \sup_{\diam(\Q)<\delta}\frac{\M(\Q)}{\vert\Q\vert}\vert\mathop{\cup}_{i\in I} \Q_i\vert\\&\le \sup_{\diam(\Q)<\delta}\frac{\M(\Q)}{\vert\Q\vert}\vert E_\delta\vert.
\end{align*}
Passing to limit $\delta\to 0$ we obtain $\M^\sharp(E)\le \omega\vert E\vert$.
\end{proof}

Let $\M^+_\ast, \M^-_\ast:\Omega\to [0,\infty]$ be the functions defined by
\begin{align*}
\M^+_\ast(x):=\limsup_{\delta \to 0}\sup_{x\in\Q\in\Cub_\delta}\frac{\M^\sharp(\Q)}{\vert\Q\vert}\mbox{ and }\M^-_\ast(x):=\liminf_{\delta \to 0}\inf_{x\in\Q\in\Cub_\delta}\frac{\M(\Q)}{\vert\Q\vert}.
\end{align*}
\begin{lemma}\label{measurability-m} Let $a,b\in \RR^+$. Then
\begin{itemize}
\item[(i)] there exists a Borel set $B_a^+\subset \{x\in \Omega: \M^+_\ast(x)\ge a\}$ such that $\vert \{x\in \Omega: \M^+_\ast(x)\le a\}\ssetminus B_a^+\vert=0$;
\item[(ii)] there exists a Borel set $B_b^-\subset \{x\in \Omega: \M^-_\ast(x)\le b\}$ such that $\vert \{x\in \Omega: \M^-_\ast(x)\le b\}\ssetminus B_b^-\vert=0$.
\end{itemize}
\end{lemma}
\begin{proof} Let us prove (ii). Let $b\in\RR^+$. Set $M_b=\{x\in \Omega: \M^-_\ast(x)\le b\}$. For each $k\in\NN^*$, consider the set
\begin{align*}
\mathcal{G}_{k}:=\left\{\overline{\Q}_\eps(x):x\in M_b, \;\eps\in]0,\mbox{${\frac1k}$}[\mbox{ and }\M(\Q_\eps(x))\le (b+\mbox{${\frac1k}$})\vert \Q_\eps(x)\vert\right\}.
\end{align*}
For each $k\ge 1$ the family $\mathcal{G}_{k}$ is a fine cover of $M_b$, and so by the Vitali covering theorem, there exists a disjointed countable family $\{\overline{\Q}_i^k\}_{i\in I_k}\subset \mathcal{G}_{k}$ such that $\vert M\ssetminus \mathop{\cup}_{i\in I_k}\Q_i^k\vert=0$. Consider the Borel set $B_b^-:=\cap_{k\ge 1}\mathop{\cup}_{i\in I_k}\Q_i^k$, then $\vert M_b\ssetminus B_b^-\vert\le \sum_{k\ge 1}\vert M_b\ssetminus \mathop{\cup}_{i\in I_k}\Q_i^k\vert=0$. If we show that $B_b^-\subset M_b$ then the proof of (ii) will be complete. Let $y\in B_b^-$. Then for every $k\ge 1$ there exists $i_k\in I_k$ such that $y\in\Q_{i_k}^k\in\Cub_{\frac1k}$ and $\M(\Q_{i_k}^k)\le (b+\frac1k)\vert \Q_{i_k}^k\vert$. It follows that
\begin{align*}
\inf_{y\in\Q\in\Cub_{\frac1k}}\frac{\M(\Q)}{\vert \Q\vert}\le \frac{\M(\Q_{i_k}^k)}{\vert \Q_{i_k}^k\vert}\le b+\frac1k,
\end{align*}
letting $k\to \infty$, we obtain that $\M^-_\ast(y)\le b$ which means that $y\in M_b$.

For the proof of (i), it is enough to remark that for $a\ssup 0$
\[
\{x\in \Omega: \M^+_\ast(x)\ge a\}=\left\{x\in \Omega: \liminf_{\delta\to 0}\inf_{x\in\Q\in\Cub_{\delta}}\frac{\vert\Q\vert}{\M^\sharp(\Q)}\le \mbox{$\frac1a$}\right\},
\]
and to apply the same reasoning as in the proof of (ii) with the necessary changes.
\end{proof}
\begin{remark} By Lemma~\ref{measurability-m}, the functions $\M^-_\ast$ and $\M^+_\ast$ are measurable. 
\end{remark}
\begin{remark}\label{strictineq} The same conclusions can be drawn if we replace large inequalities with strict inequalities in the Lemma~\ref{measurability-m}, indeed, it suffices to see for instance that
\begin{align*}
\left\{x\in \Omega: \M^+_\ast(x)\ssup a\right\}=\mathop{\cup}_{n\ge 1} \left\{x\in \Omega: \M^+_\ast(x)\ge a+\mbox{$\frac1n$}\right\}.
\end{align*}
\end{remark}
We denote by $\overline{\M^\sharp}$ the set function 
\begin{align*}
\overline{\M^\sharp}(E)=\inf\big\{\M^\sharp(O): E\subset O,\; O\mbox{ open }\big\} \;\mbox{ for all }E\subset\RR^d.
\end{align*}
\begin{lemma}\label{galere} If $\omega\sinf\infty$ then $\overline{\M^\sharp}(K)={\M^\sharp}(K)$ for all compact $K\subset\Omega$.
\end{lemma}
\begin{proof} Fix a compact set $K\subset \Omega$. Note that by Lemma~\ref{outermeasure} (ii) we have $\overline{\M^\sharp}(K)\ge \M^\sharp(K)$. So it remains to prove the reverse inequality $\overline{\M^\sharp}(K)\le \M^\sharp(K)$. 

By Lemma~\ref{outermeasure} (iii) we have $\M^\sharp(K)\le \omega\vert K\vert\le \omega\vert\Omega\vert\sinf\infty$. Let $O\subset \Omega$ be an open set such that $O\supset K$. For each $n\in\NN^*$ such that $n\ge n_0$ where ${n_0}:={\rm Ent}\left(\left(\diam(O)-\diam(K)\right)^{-1}\right)+1$ (where {\rm Ent}(r) denotes the integer part of the real number $r$) there exists $\{\Q_j^n\}_{j\ge 1}\subset \FF_{\frac1n}(K)$ such that
\begin{align*}
\infty\ssup{\M^\sharp}(K)+\frac1n\ge \sum_{j\ge 1}\M(\Q_j^n).
\end{align*}
Note that $\mathop{\cup}_{j\ge 1}\Q_j^n\subset \{x\in\RR^d:\dist(x,K)\sinf\frac1n\}\subset O$ because $n\ge n_0$ and $\Q_j^n\cap K\not=\emptyset$ for all $j\ge 1$.

Fix $n\ge n_0$. Then there exists an increasing sequence $\{j_s\}_{s\ge 1}$ such that $\sup_{s\ge 1}j_s=\infty$ and $\alpha_s:=\sum_{j\ge j_s}\M(\Q_j^n)\le \frac1s$ for all $s\ge 1$. Fix $s\in\NN^*$. For the open set $O\ssetminus \mathop{\cup}_{1\le j\le j_s}\overline{\Q}^n_j$ we use the Vitali covering theorem to find a disjointed countable family of closed cubes $\{\overline{\Q}_i^n\}_{i\in I}$ such that $\diam(\Q_i^n)\sinf\frac1n$,
\[
\left\vert \left(O\ssetminus \mathop{\cup}_{1\le j\le j_s}\overline{\Q}_j^n\right)\ssetminus \mathop{\cup}_{i\in I}\overline{\Q}_i^n\right\vert=0, \mbox{ and } \mathop{\cup}_{i\in I}\overline{\Q}_i^n\subset O\ssetminus \mathop{\cup}_{1\le j\le j_s}\overline{\Q}_j^n.
\]
It is easy to see that the countable family 
\[
\left\{\overline{\Q}_k^n\right\}_{k\in D}:=\left\{\overline{\Q}_i^n:i\in I\right\}\cup\left\{\overline{\Q}_j^n:1\le j\le j_s\right\}\in\FF_{\frac1n}(O).
\] 
If $\omega_{\frac1n}:=\sup_{\Q\subset\Omega,\;\diam(\Q)<\frac1n}\frac{\M(\Q)}{\vert \Q\vert}$ then
\begin{align*}
\M^{\frac1n}(O)-\M^\sharp(K)&\le \sum_{k\in D}\M(\Q_k^n)-\sum_{j\ge 1}\M(\Q_j^n)+\frac1n\\
&=\sum_{i\in I}\M(\Q_i^n)-\alpha_s+\frac1n\\
&\le \omega_{\frac1n}\vert O\ssetminus \mathop{\cup}_{1\le j\le j_s}\overline{\Q}_j^n\vert-\alpha_s+\frac1n.
\end{align*}
Passing to the limit $s\to \infty$ we obtain $\M^{\frac1n}(O)-\M^\sharp(K)\le\omega_{\frac1n}\vert O\ssetminus \mathop{\cup}_{j\ge 1}\overline{\Q}_j^n\vert+\frac1n $. If $E:=\cap_{n\ge n_0}\mathop{\cup}_{j\ge 1}\Q_j^n$ then $\vert K\ssetminus E\vert=0$, indeed, we have 
\begin{align*}
\vert K\ssetminus E\vert\le\sum_{n\ge n_0}\vert K\ssetminus \mathop{\cup}_{j\ge 1}\overline{\Q}_j^n\vert=0.
\end{align*}
Letting $n\to \infty$ it follows that $\M^{\sharp}(O)-\M^\sharp(K)\le \omega\vert O\ssetminus E\vert$. Therefore
\begin{align*}
\M^\sharp(O)&\le \M^\sharp(K)+\omega\left(\vert (O\ssetminus K)\ssetminus E\vert+\vert K \ssetminus E\vert\right)\\
&\le  \M^\sharp(K)+\omega\vert O\ssetminus K\vert.
\end{align*}
Since the open set $O$ containing $K$ is arbitrary, by the outer regularity of the Lebesgue measure we obtain $\overline{\M^\sharp}(K)\le \M^\sharp(K)$, and the proof is complete.
\end{proof}

\begin{lemma}\label{gmt-classics} Let $a,b\ssup 0$. Let $E\subset\Omega$ be an arbitrary set. 
\begin{enumerate}
\item[(i)] If $E\subset \{x\in \Omega: \M^+_\ast(x)\ssup a\}$ then $\overline{m^\sharp}(E)\ge a\vert E\vert$;
\item[(ii)] If $E\subset \{x\in \Omega: \M^-_\ast(x)\sinf b\}$ then $m^\sharp(E)\le b\vert E\vert$.
\end{enumerate}
\end{lemma}
\begin{proof} We start by the proof of (i). Fix $a\ssup0$ and let $E\subset \{x\in \Omega: \M^+_\ast(x)\ssup a\}$. Let $O$ be an open set of $\Omega$ such that $O\supset E$. We can rewrite
\[
\{x\in \Omega: \M^+_\ast(x)\ssup a\}=\left\{x\in \Omega: \liminf_{\delta\to 0}\inf_{x\in\Q\in\Cub_{\delta}}\frac{\vert\Q\vert}{\M^\sharp(\Q)}\sinf \mbox{$\frac1a$}\right\}.
\]
Fix $\delta\ssup 0$ and consider the family of closed cubes
\begin{align*}
\G_\delta:=\left\{\overline{\Q}_\eps(x):x\in E,\;\Cub_\delta\ni{\Q}_\eps(x)\subset O\mbox{ and }\vert \Q_\eps(x)\vert\le\mbox{$\frac{1}{a}$}\M^\sharp(\Q_\eps(x))\right\}.
\end{align*}
The family $\G_\delta$ is a fine covering of $E$. By the Vitali covering theorem, there exists a disjointed countable family $\{\overline{\Q}_i\}_{i\in I}\subset\G_\delta$ such that $\vert E\ssetminus \mathop{\cup}_{i\in I}\Q_i\vert=0$. For each $\eps\ssup 0$ there exists a finite set $I_\eps\subset I$ such that $\vert E\ssetminus \mathop{\cup}_{i\in I_\eps}\Q_i\vert\sinf\eps$. Then by using Lemma~\ref{outermeasure} (i) 
\begin{align*}
\vert E\vert-\eps=\vert E\cap \mathop{\cup}_{i\in I_\eps}\Q_i\vert\le \sum_{i\in I_\eps}\vert\Q_i\vert\le \frac{1}{a}\sum_{i\in I_\eps}\M^\sharp(\Q_i)=\M^\sharp(\mathop{\cup}_{i\in I_\eps}\Q_i)\le \M^\sharp(O).
\end{align*}
The proof of (i) is complete since the open set $O$ which contains $E$ is arbitrary.

It remains to prove (ii). For each $\delta\ssup 0$ consider the set
\begin{align*}
\G_\delta:=\big\{\overline{\Q}_\eps(x)\in\Cub_\delta: {\Q}_\eps(x)\in\Cub_\delta,\; x\in E\mbox{ and }\M(\Q_\eps(x))\le b\vert\Q_\eps(x)\vert\big\}.
\end{align*} 
It is a fine cover of $E$, i.e., $\inf\{\diam(\Q):\Q\in\G_\delta\}=0$. Then there exists a disjointed countable subfamily $\{\Q_i\}_{i\in I}\subset\G_\delta$ such that $\vert E\setminus \mathop{\cup}_{i\in I}\Q_i\vert=0$ and $\sum_{i\in I}\vert \Q_i\vert\le \vert E\vert+\delta$ (see \cite[Theorem 2.2, p. 26]{mattilabook}), so $\{\Q_i\}_{i\in I}\in \FF_\delta(E)$. It follows that
\begin{align*}
m^\sharp(E)\le\sum_{i\in I}\M(\Q_i)\le \sum_{i\in I}b\vert \Q_i\vert\le b\vert E\vert+b\delta,
\end{align*}
 and the proof of (ii) is complete by letting $\delta\to 0$.
\end{proof}
\begin{lemma}\label{existlimitcor} If $\omega\sinf\infty$ then $\M^+_\ast(x)\le \M^-_\ast(x)$ a.e. in $\Omega$.
\end{lemma}
\begin{proof} Fix $a,b\in\mathbb{Q}$ such that $a\ssup b\ssup 0$. Consider the following set
\begin{align*}
S_{a,b}:=\big\{x\in\Omega: \M^-_\ast(x)\sinf b\sinf a\sinf \M^+_\ast(x)\big\}.
\end{align*}
By Remark~\ref{strictineq} and Lemma~\ref{measurability-m} there exists a Borel set $B_{a,b}$ such that $B_{a,b}\subset S_{a,b}$ and $\vert S_{a,b}\ssetminus B_{a,b}\vert=0$. Fix $\eps\ssup 0$. Since the Lebesgue measure is inner regular, choose a compact set $K_\eps\subset B_{a,b}$ such that $\vert B_{a,b}\ssetminus K_\eps\vert\sinf\eps$. From Lemma~\ref{galere} we have $\overline{\M^\sharp}(K_\eps)={\M^\sharp}(K_\eps)$ since $\omega\sinf\infty$. Using Lemma~\ref{gmt-classics} we obtain $a \vert K_\eps\vert\le m^\sharp(K_\eps)\le b \vert K_\eps\vert$. Therefore $\vert K_\eps\vert=0$ since $b\sinf a$. Hence $\vert S_{a,b}\vert=\vert B_{a,b}\ssetminus K_\eps\vert+\vert K_\eps\vert+\vert S_{a,b}\ssetminus B_{a,b}\vert\sinf\eps$, and $\vert S_{a,b}\vert=0$ by letting $\eps\to 0$. Now, the set where $\M^+_\ast$ is greater than $\M^-_\ast$ is a countable union of negligible sets, i.e.,
\begin{align*}
\big\{x\in\Omega:  \M^-_\ast(x)\sinf \M^+_\ast(x)\big\}=\bigcup_{0<b<a,\;(a,b)\in \mathbb{Q}^2} S_{a,b},
\end{align*}
and the proof is complete.
\end{proof}

\bigskip

\noindent{\bf Proof of Proposition~\ref{existlimit}}. Using \eqref{hyp-import} and the definitions of $\M^+_\ast$ and  $\M^-_\ast$ we have for every $x\in\Omega$
\begin{align*}
\M_\ast^-(x)\le\liminf_{\delta\to 0}\frac{\M(\Q_\delta(x))}{\delta^d}\le \limsup_{\delta\to 0}\frac{\M(\Q_\delta(x))}{\delta^d}\le \M_\ast^+(x).
\end{align*}
By Lemma~\ref{existlimitcor} we obtain
\begin{align*}
\M_\ast^-(x)=\M_\ast^+(x)=\lim_{\delta\to 0}\frac{\M(\Q_\delta(x))}{\delta^d}\mbox{ a.e. in }\Omega
\end{align*}
which completes the proof.\hfill$\blacksquare$

\medskip

If $L:\Omega\times\MM^{m\times d}\to [0,\infty]$ is a Borel measurable integrand then for each $\xi\in\MM^{m\times d}$ we denote by $\M_\xi:\Cub\to [0,\infty]$ the set function defined by
\begin{align*}
\M_\xi(\Q)=\inf\left\{\int_\Q L(x,\xi+\nabla \varphi(x))dx:\varphi\in W^{1,p}_0(\Q;\RR^m)\right\}.
\end{align*}
\subsection{Proof of Proposition~\ref{exist-ZL}} The Proposition~\ref{exist-ZL} follows from Proposition~\ref{existlimit} by noticing that 
\begin{align*}
\Z L(x,\xi)=\liminf_{\eps\to 0}\frac{\M_\xi(\Q_\eps(x))}{\eps^d},
\end{align*}
and by using the following result.
\begin{lemma}\label{mastinquelity} If \ref{H4} holds then for every $x\in\Omega$ and every $\xi\in\dom L(x,\cdot)$ we have
\begin{align}\label{eq01:mastinquelity}
\limsup_{\eps\to 0}\frac{\M_\xi(\Q_\eps(x))}{\eps^d}\le\limsup_{\eps\to 0}\frac{\M^{\sharp}_\xi(\Q_\eps(x))}{\eps^d}.
\end{align}
\end{lemma}
\begin{proof} Fix $\eps\in ]0,1[$ and $s\ssup 1$. Fix $x\in \Omega^\prime$ where $\Omega^\prime=\{x\in\Omega: \dom L(x,\cdot)\subset\Lambda_L(x)\}$ which satisfies $\vert\Omega\ssetminus\Omega^\prime\vert=0$ since \ref{H4}. Fix $\xi\in\dom L(x,\cdot)$ and fix $\delta\in ]0,2(s-1)\eps^{1-d}[$. Choose $\{\overline{\Q}_i\}_{i\ge 1}\in \FF_{\frac{\delta\eps^d}{2}}(\Q_\eps(x))$ such that $\vert \Q_\eps(x)\ssetminus \mathop{\cup}_{i\ge 1}\Q_i\vert=0$, $\Q_i\cap \Q_\eps(x)\not=\emptyset$ for all $i\ge 1$, and
\begin{align}\label{eq1:mastinquelity}
 \sum_{i\ge 1} \M_\xi(\Q_i)\le \frac{\delta\eps^d}{2}+\M_\xi^\sharp(\Q_\eps(x)).
\end{align}
If $O_\delta=\mathop{\cup}_{i\ge 1}\Q_i$ then $\Q_\eps(x)\subset O_\delta\subset \Q_{s\eps}(x)$. Indeed, on one hand we have $O_\delta\subset \{y\in\Omega:\dist(y,\Q_\eps(x))\sinf\frac{\delta\eps^d}{2}\}$ and $\frac{\delta\eps^d}{2}+\eps\le s\eps$, thus $O_\delta\subset \Q_{s\eps}(x)$. On the other hand $\Q_\eps(x)\ssetminus \overline{O_\delta}$ is open and $\vert \Q_\eps(x)\ssetminus \overline{O_\delta}\vert\le \vert \Q_\eps(x)\ssetminus {O_\delta}\vert=0$, therefore $\Q_\eps(x)\ssetminus \overline{O_\delta}=\emptyset$. It follows that $\Q_\eps(x)\subset \overline{O_\delta}$ and so $\Q_\eps(x)\subset {O_\delta}$.

For each $i\ge 1$ there exists $\varphi_\delta^i\in {W}^{1,p}_0(\Q_i;\RR^m)$ such that
\begin{align}\label{eq2:mastinquelity}
 \int_{\Q_i}L(y,\xi+\nabla\varphi_\delta^i)dy\le \frac{\delta\eps^d}{2^{i+1}}+\M_\xi(\Q_i).
\end{align}
Define $\varphi_\delta:=\sum_{i\ge 1}\varphi_\eps^i\mathbb{I}_{\Q_i}\in  {W}^{1,p}_0(O_\delta;\RR^m)$. By taking account of \eqref{eq1:mastinquelity} we have
\begin{align}\label{eq3:mastinquelity}
\int_{O_\delta}L(y,\xi+\nabla\varphi_\delta)dy= \sum_{i\ge 1}  \int_{\Q_i}L(y,\xi+\nabla\varphi_\delta^i)dy&\le \frac{\delta\eps^d}{2}+\sum_{i\ge 1} \M_\xi(\Q_i)\\
&\le \delta\eps^d+\M^\sharp_\xi(\Q_\eps(x)).\notag
\end{align}
The function $\varphi_\delta$ also belongs in ${W}^{1,p}_0(\Q_{s\eps}(x);\RR^m)$, and moreover
\begin{align}\label{eq4:mastinquelity}
\int_{O_\delta}L(y,\xi+\nabla\varphi_\delta)dy&\ge \int_{\Q_{s\eps}(x)}L(y,\xi+\nabla\varphi_\delta)dy-\int_{\Q_{s\eps(x)}\ssetminus \Q_\eps(x)}L(y,\xi)dy\\
&\ge \M_\xi(\Q_{s\eps}(x))-\int_{\Q_{s\eps(x)}\ssetminus \Q_\eps(x)}L(y,\xi)dy.\notag
\end{align}
From \ref{H4} we have
\begin{align*}
\limsup_{\eps\to 0}\frac{1}{\eps^d}\int_{\Q_{s\eps}(x)\ssetminus \Q_\eps(x)}L(y,\xi)dy&=\limsup_{\eps\to 0}\left\{s^d\fint_{\Q_{s\eps}(x)}\!\!L(y,\xi)dy-\fint_{\Q_{\eps}(x)}\!\! L(y,\xi)dy\right\}\\
&\le (s^d-1)L(x,\xi).
\end{align*}
From \eqref{eq3:mastinquelity} and \eqref{eq4:mastinquelity} it holds since $s\ssup 1$
\begin{align}\label{eq5:mastinquelity}
\limsup_{\eps\to 0}\frac{\M_\xi(\Q_{s\eps}(x))}{(s\eps)^d}\le (s^d-1)L(x,\xi)+\delta+\limsup_{\eps\to 0}\frac{\M^\sharp_\xi(\Q_\eps(x))}{\eps^d}.
\end{align}
But again since $s\ssup 1$ we have
\begin{align*}
\limsup_{\eps\to 0}\frac{\M_\xi(\Q_{s\eps}(x))}{(s\eps)^d}&=\lim_{\eps \to 0}\sup_{\eta\in ]0,\eps[}\frac{\M_\xi(\Q_{s\eta}(x))}{(s\eta)^d}\\
&=\lim_{\eps \to 0}\sup_{\eta^\prime\in ]0,s\eps[}\frac{\M_\xi(\Q_{\eta^\prime}(x))}{(\eta^\prime)^d}\\
&\ge \lim_{\eps \to 0}\sup_{\eta^\prime\in ]0,\eps[}\frac{\M_\xi(\Q_{\eta^\prime}(x))}{(\eta^\prime)^d}=\limsup_{\eps\to 0}\frac{\M_\xi(\Q_{\eps}(x))}{\eps^d}.
\end{align*}
Therefore from \eqref{eq5:mastinquelity} we obtain
\begin{align*}
\limsup_{\eps\to 0}\frac{\M_\xi(\Q_{\eps}(x))}{\eps^d}\le (s^d-1)L(x,\xi)+\delta+\limsup_{\eps\to 0}\frac{\M^\sharp_\xi(\Q_\eps(x))}{\eps^d},
\end{align*}
letting $s\to 1$ and $\delta\to 0$ we obtain \eqref{eq01:mastinquelity} and the proof is complete.
\end{proof}

\bibliographystyle{alpha}

\end{document}